\documentclass[12pt]{article}
\usepackage{amscd,amssymb,amsmath,verbatim,amsthm}

\usepackage{color}

\newcommand{\CC}{\mathbb C}

\newcommand{\NN}{\mathbb N}
\newcommand{\RR}{\mathbb R}
\newcommand{\ZZ}{\mathbb Z}
\newcommand{\del}{\partial}

\newcommand{\e}{\varepsilon}
\newcommand{\K}{{\mathrm K}}
\newcommand{\al}{\alpha}
\newcommand{\be}{\beta}
\newcommand{\om}{\omega}
\newcommand{\Ric}{\mathrm{Ric}}
\newcommand{\Diff}{\mathrm{Diff}}

\newcommand{\calC}{{\mathcal C}}
\newcommand{\calD}{{\mathcal D}}
\newcommand{\calE}{{\mathcal E}}
\newcommand{\calF}{{\mathcal F}}
\newcommand{\calG}{{\mathcal G}}
\newcommand{\calH}{\mathcal H}

\newcommand{\calL}{{\mathcal L}}
\newcommand{\calM}{{\mathcal M}}

\newcommand{\calO}{{\mathcal O}}
\newcommand{\calP}{{\mathcal P}}
\newcommand{\calS}{{\mathcal S}}
\newcommand{\calT}{{\mathcal T}}
\newcommand{\calU}{{\mathcal U}}
\newcommand{\calV}{{\mathcal V}}
\newcommand{\calW}{{\mathcal W}}

\newcommand{\tr}{{\mathrm{tr}}\,}
\newcommand{\olg}{\overline{g}}

\newcommand{\circR}{\overset{\circ}{R}}

\newcommand{\frakc}{\mathfrak c}
\newcommand{\frakp}{\mathfrak p}

\newcommand{\Si}{\Sigma}
\newcommand{\Sip}{\Si_{\mathfrak p}}
\newcommand{\Stt}{\calS_{\mathrm{tt}}}
\newcommand{\cc}{\mathrm{cc}}
\newcommand{\conic}{\mathrm{conic}}

\newcommand{\sing}{\mathrm{sing}\,}
\newcommand{\CM}{\calC\calM}
\newcommand{\HCM}{\calH\calC}
\newcommand{\hc}{\mathrm{hc}}
\newcommand{\id}{\operatorname{id}}
\newcommand{\sign}{\operatorname{sign}}

\newtheorem{theorem}{Theorem}
\newtheorem{proposition}{Proposition}
\newtheorem{corollary}{Corollary}
\newtheorem{lemma}{Lemma}
\theoremstyle{definition}
\newtheorem{definition}{Definition}
\theoremstyle{remark}
\newtheorem*{remark}{Remark}

\title{Teichm\"uller theory for conic surfaces}
\author{Rafe Mazzeo \\ Stanford University \and Hartmut Weiss 
\\ Christian-Albrechts-Universit\"at zu Kiel}

\date{}

\begin{document}

\maketitle

\section{Introduction}
Let $\Si$ be a compact oriented surface of genus $\gamma$, $\frakp = \{p_1, \ldots, p_k\} \subset \Si$ a fixed 
collection of $k$ distinct points, and set $\Sip = \Si \setminus \frakp$. A conformal class (or conformal structure) $\frakc$ on $\Sip$ is 
the set of all multiples $e^{2\phi} g$ of a given metric $g$; this determines a holomorphic structure on the open surface and vice versa.  
Our convention is that a holomorphic structure in a neighborhood of a puncture is one which extends over the 
puncture. The Teichm\"uller space $\calT_{\gamma,k}$ is the space of conformal classes on $\Si$ modulo diffeomorphisms of 
$\Si$ isotopic to the identity which fix $\frakp$.  The uniformization theorem states that each conformal class contains 
a constant curvature metric which is complete on $\Sip$. When $2-2\gamma-k < 0$, this is a finite area hyperbolic metric
and is unique (if the curvature is normalized to equal $-1$). In the few remaining cases ($(\gamma,k) = (0,0), (0,1), (0,2)$ or $(1,0)$), 
uniformizing metrics are either spherical or flat. 

In this paper we consider, as a different canonical choice, the constant curvature metrics with conic singularities 
at $\frakp$.  Thus at each singular point $p_j$ we assume that $(\Si, g)$ is asymptotic 
to a standard cone of angle $\theta_j$; writing $\theta_j = 2\pi (1+\beta_j)$, we refer to $\beta_j$ as 
the cone angle parameter at that point. We restrict attention in this paper to the case where all $\theta_j \in (0,2\pi)$, or 
equivalently $\beta_j \in (-1,0)$, and write $\vec \beta = (\be_1, \ldots, \be_k)$. 

There is a complete existence theory for conic constant curvature metrics with all cone angles less than $2\pi$.  To 
state it, define the quantity
\begin{equation}
\chi(\Sip, \vec\be) = \chi(\Si) + \sum_{j=1}^k \be_j 
\label{cec}
\end{equation}
and, when $\chi(\Si) = 2$ so $\Si = S^2$, and $k \geq 3$, the so-called {\it Troyanov region} 
\begin{equation}
\calF = \{\vec \beta \in (-1,0)^k :\, \sum_{i \neq j} \be_i < \be_j \ \ \forall\, j \}.
\label{Troy}
\end{equation}
\begin{theorem}{\rm \cite{Mc, Tr, LT}}
Let $\Sip$ be a punctured compact Riemann surface, $\olg$ a smooth metric on $\Si$ and $\vec \be \in (-1,0)^k$ a collection 
of cone angle parameters. If $\chi(\Sip, \vec\be) \leq 0$, then there exists a conic metric with
constant curvature $K \leq 0$ and area $1$ in the conformal class of $\olg$ and with specified cone angle parameters $\vec\be$;
this metric is unique. If $\chi(\Sip, \vec\be) > 0$ and $k > 2$, then there exists a unique such metric if and only if $\vec\be \in \calF$; 
if $k = 2$, then such a metric exists if and only if $\be_1 = \be_2$, and it is unique up to conformal dilations.  
\label{fundexist}
\end{theorem}
The fact that the sign of $K$ agrees with that of $\chi(\Sip, \vec\be)$ is a consequence of the Gauss-Bonnet theorem.
Existence and uniqueness when $K \leq 0$ is due to McOwen. Existence and uniqueness when $K > 0$ were proved 
by Troyanov and Luo-Tian, respectively, and the fact that the Troyanov condition is necessary for existence has been observed
by several people, and follows from an appealing geometric argument presented in \S 2 below.   

Observe also that when $\Si = S^2$, $k=2$ and $\be_1 = \be_2$, the solutions are the rotationally invariant models
$dr^2 + (1+\be_1)^2 \sin^2 r dy^2$. When $\be_1 \neq \be_2$ the closest thing to a canonical metric 
is a Ricci soliton, see \cite{PSS} and \cite{MRS}.

The variational method employed by Troyanov relies on the problem being coercive when $\vec\theta \in \calF$. This fails when any of 
the cone angles are greater than $2\pi$. Early work in this more general case by Umehara and Yamada \cite{UY} when $k=3$ and 
$K = 1$ and Eremenko \cite{Er} relies on complex analytic ideas. Some general existence results have been obtained using minimax 
theory by Carlotto and Malchiodi \cite{CM1}, \cite{CM2}, but these metods require that $\gamma \geq 1$. Quite recently, 
Mondello and Panov \cite{MonP} have established a dramatic existence theorem using entirely synthetic geometric methods 
when $\Si = S^2$. However, many interesting problems remain open.

In this paper we develop a systematic deformation theory for conic constant curvature metrics when all cone angles are less than $2\pi$;
in particular, we define and study the Teichm\"uller space $\calT^{\conic}_{\gamma,k}$  of conic constant curvature metrics on a surface
of genus $\gamma$ with $k$ conic points. 
The methods here are adopted from higher dimensional global analysis, generalizing Tromba's approach \cite{Tro} to the 
study of the standard Teichm\"uller space $\calT_\gamma$. The main new ingredient is the theory of elliptic conic operators. 
This is the simplest setting where the study of canonical metrics on singular spaces can be carried out in complete detail, 
and is also the one complex dimension version of the deformation theory of K\"ahler-Einstein edge metrics, cf.\ \cite{Don}, \cite{JMR}. 

We now describe our results more carefully. Let $\CM(\Sip)$ denote the set of all metrics on $\Si$ with conic singularities 
at the points in $\frakp$ (and with some fixed H\"older regularity to be specified later), and if $\vec \beta$ is
a $k$-tuple of cone angle parameters, then $\CM(\Sip, \vec\beta)$ is the subspace of conic metrics with those
specified cone angles. Inside of these are $\CM_{\cc}(\Sip)$ and $\CM_{\cc}(\Sip, \vec\beta)$ of constant 
curvature metrics. Note that we allow the curvature to be any real number, not just $-1,0,+1$, since the transitions 
between the hyperbolic, flat and spherical cases as the cone angles vary are of interest. 
Now consider the `gauge group' $\calG$ of diffeomorphisms of $\Sip$, again with some fixed regularity and decay 
properties near $\frakp$. This acts smoothly on $\CM_{\cc}(\Sip ,\vec \beta)$, but not on $\CM(\Sip, \vec \beta)$. 

There is another useful action which is not so canonical. This is of $\RR^k$ on $\CM(\Sip, \vec\beta)$ by localized 
conformal rescaling by different constants near each cone point. Note that such constant rescalings leave the cone angles unchanged.
To define this, choose cutoff functions $\chi_j$ with disjoint supports and with $\chi_j = 1$ near $p_j$. and set
\[
(\vec \lambda, g)  \mapsto  \exp( \textstyle\sum_j \chi_j \lambda_j) g.
\]

We shall prove the following results. The weighted H\"older spaces appearing here are defined in \S 4.1, and as for all the results
stated in this introduction, the precise statements with regularity assumptions are given in \S 6 and \S 8.
\begin{theorem}
The spaces $\CM_\cc(\Sip)$ and $\CM_{\cc}(\Sip, \vec\beta)$ are smooth Banach submanifolds in $\CM(\Sip)$.  
For any $\vec{\beta} \in \calF$, 
\[
\begin{array}{rcl}
(\RR^k \oplus r^\nu \calC^{2,\alpha}_b(\Sip)) \times \CM_{\cc}(\Sip ,\vec\beta) & \longrightarrow &\CM(\Sip, \vec\beta) \\
(\vec\lambda, \phi,g) & \longmapsto & e^{2( \sum_j \chi_j \lambda_j + \phi)}g 
\end{array}
\]
is a diffeomorphism.  
\end{theorem}
This is simply Poincar\'e's uniformization theorem for conic metrics, with the additional conclusion that the construction 
depends smoothly on the metric. 
\begin{theorem}
Fix any smooth element $g \in \CM_{\cc}(\Sip)$. Then there is an embedded $(6\gamma-6 + 3k)$-dimensional submanifold 
$\calS \subset \CM_{\cc}(\Sip)$ passing through $g$, such that 
\[
\mathrm{Diff}\, (\Sip) \times \calS \hookrightarrow \CM_{\cc}(\Sip), \quad (F, g) \mapsto F^* g
\] 
is a local diffeomorphism. 
\end{theorem}
\begin{theorem}
The actions of $\mathrm{Diff}(\Sip)$ and $\mathrm{Diff}_0(\Sip)$ on $\CM_{\cc}(\Sip, \vec\beta)$ are proper and 
proper and free, respectively. 
\end{theorem}

When the cone angle is zero, the metrics are complete finite area and hyperbolic, and the analogues of these
results are classical.   Some results in this direction, along the lines of the approach in \cite{Tro}, are
contained in \cite{Zeit}.  It turns out to be quite simple to extend the analysis to that case, and we prove the
\begin{theorem}
The direct analogues of Theorems 2, 3 and 4 hold for asymptotically cusp metrics on $\Sip$, with $\CM_\cc(\Sip; \vec\beta)$
replaced by the space of complete, finite area hyperbolic metrics on the punctured surface.
\end{theorem}
There is a further relationship between these two settings since conic constant curvature metrics 
converge to complete finite area hyperbolic metrics as the cone angle parameters $\beta_j \searrow -1$. 
Thus $\CM_{\cc}(\Sip)$ interpolates between the unpunctured and punctured Teichm\"uller spaces, $\calT_\gamma$ 
and $\calT_{\gamma,k}$, and hence may be useful in comparing quantities between these two spaces.  We have in 
mind the Weil-Petersson metrics, which have already been studied in the conic setting by Schumacher and 
Trapani \cite{ST}, and the determinant functional $g \mapsto \det' \Delta_g$, which is analyzed for conic metrics 
by Kokotov \cite{Kok}.  We shall return to this circle of questions elsewhere. 

A technical result gives the sharpest regularity of the optimal representatives of elements in the 
conic Teichm\"uller space. 
\begin{theorem}
Let $g$ be a polyhomogeneous conic metric.  Then every metric in the slice $\calS$ based at $g$
is also polyhomogeneous. 
\end{theorem}
A function is said to be polyhomogeneous at a conic point if, in polar coordinates around that point, it has a 
complete asymptotic expansion with smooth coefficients; a metric is polyhomogeneous if its coefficients
have this property, or equivalently, if it is conformal with a polyhomogeneous factor, to a background
smooth metric, see \S 4.2.   This type of regularity is particularly useful in many calculations. 

There is now a growing body of work on conic metrics on surfaces and their analogues in higher dimensions.
We mention in particular a Ricci-flow approach to the existence of conic metrics by the first author,
Rubinstein and Sesum \cite{MRS}, cf.\ also \cite{PSSW}, and the existence of K\"ahler-Einstein edge metrics in 
higher dimensions \cite{JMR}. These are simply the papers closest in spirit to the present one, and we do 
not attempt to list any further works in the increasingly large literature in this area. 

The plan of this paper is as follows: we begin in the next section by defining $\CM(\Sip)$ and discussing the basic 
geometric properties of conic metrics. This is followed by a review of the global analytic framework used here, i.e., the curvature 
function and its linearization, the Bianchi gauge, the gauged curvature equation, etc., and the general theory of elliptic conic 
operators. In \S 5 we calculate the indicial roots for the relevant operators, and prove the main results in \S 6; the case
of zero cone angle (complete hyperbolic metrics) is handled in \S 7. Finally in \S 8 we describe certain global features, including
the transition interfaces as the constant curvature changes from negative to positive. 

\medskip

\noindent{\bf Acknowledgements:}  This paper was started several years ago during an extended visit by the second author
to Stanford University, funded by the DFG. The authors also wish to acknowledge the hospitality of several locations where
the work progressed, including the Mathematisches Forschungsinstitut Oberwolfach, the American Institute of Mathematics 
and most recently, the Isaac Newton Institute.  The first author was supported in the later stages of this work by NSF DMS-1105050.  
We profited from discussing these and related issues with many people over the years. 

\section{Geometry of conic metrics}
This section collects a number of geometric facts about the class of conic metrics, focusing on those with cone angles less than $2\pi$. 

\subsection{The models}
We frequently refer to the model two-dimensional conic metrics with constant curvature $K$ and cone parameter $\be$: 
\begin{equation}
g_{\be,K} = \begin{cases} dr^2 + (1+ \be)^2|K|^{-1} \sinh^2 |K|^{1/2} r \, dy^2, \quad & K < 0, \\
dr^2 + (1+ \be)^2 dy^2, \quad & K = 0 \\  dr^2 + (1+ \be)^2 K^{-1} \sin^2 (K^{1/2} r) \, dy^2, \quad & K > 0,
\end{cases}
\label{models1}
\end{equation}
where $y \in S^1 = \RR/2\pi \ZZ$. In each case, the cone angle is $\theta = 2\pi (1 + \be)$.  The conformal version  in the flat case is
\begin{equation}
g_{\be,0} = |z|^{2\beta} |dz|^2,   \quad z \in\CC\cong\RR^2.
\label{eq:confform}
\end{equation}
The corresponding expressions when $K \neq 0$ are more complicated, and not needed later. To see that this is equivalent to \eqref{models1},
write $z = \rho e^{iy}$ and set $r = \rho^{\be+1}/(\be + 1)$.  We remark again that $\theta \in (0,2\pi) \Longleftrightarrow \beta \in (-1,0)$. 

For simplicity, we henceforth mostly write $g_\be$ instead of $g_{\be,0}$. 

\subsection{The space of  conic metrics}
We next generalize these models and describe the corresponding local expressions for general conic metrics in two dimensions.

Let $\calU$ be the unit ball, and suppose that $z, \rho, y$ and $r$ have the same meanings as above. Fix a weight parameter 
$\nu > 0$. A conic metric of order $\nu$ in $\calU$ is one which takes either of the equivalent forms
\begin{equation}
e^{2\phi}|z|^{2\beta} F^* |dz|^2, \ \ \mbox{or}\ \ dr^2 + (1+\be)^2 r^2\, dy^2 + h, 
\label{localcm}
\end{equation}
where $F = \exp(X)$ is a diffeomorphism, $h = a \, dr^2 + 2b \, rdrdy + c \, r^2 dy^2$, and $X, \phi, a, b, c$ all decay like $r^\nu$.
The regularity assumptions are specified precisely in \S 6, and we prove there that these two formulations are equivalent.

Let $\Si$ be a compact surface of genus $\gamma$, with $\frakp = \{p_1, \ldots, p_k\} \subset \Si$ a collection of distinct points. 
Fix a (smooth) background metric $\olg$, which for convenience we assume is flat near each $p_j$. 
Choose a smooth positive function on $\Sip := \Si \setminus \frakp$ which
equals the $\olg$-distance near each $p_j$, and set $\calU_j = \{q: d_{\olg}(g,p_j) < \e\}$.  We shall also fix a collection
of cutoff functions $\chi_j \in \calC^\infty_0(\calU_j)$ with $\chi_j = 1$ in a neighborhood of $p_j$. 
The space $\CM(\Sip)$ consists of all metrics $g$ on $\Sip$ which can be written as in \eqref{localcm} in each $\calU_j$. 
We omit $\nu$ in this notation, and of course we have not specified the regularity of the coefficients. We refer to \S 6.1 for 
precise statements. 


\subsection{Gauss-Bonnet and Troyanov constraints}\label{Troyanov_region}
The Gauss-Bonnet formula yields a relationship between the cone angles, the cardinality of $\frakp$, the
Euler characteristic of $\Si$ and the total integral of the Gaussian curvature $K^g$. Indeed, applying the 
ordinary Gauss-Bonnet formula on $\Sip \setminus \cup_j B_\e(p_j)$ and letting $\e \to 0$ yields 
\begin{equation}
\frac{1}{2\pi}\int_{\Sip} K^g\, dA^g = \chi(\Sip,\vec\be)
\label{eq:GB1}
\end{equation}
where the right hand side was defined in \eqref{cec}. In particular, if $g$ has constant curvature $\K$, then 
\begin{equation}
\K \, \times \mbox{Area}(\Sip,g) = 2\pi \chi(\Sip,\vec\be), 
\label{eq:GB2}
\end{equation}
so $\chi(\Sip,\vec\be)$ has the same sign as $\K$. As noted earlier, \eqref{eq:GB2} is also sufficient for 
existence in a given conformal class when $\K \leq 0$, but the additional Troyanov condition \eqref{Troy} is needed
when $\K > 0$. 

We now show the necessity of \eqref{Troy} for the existence of spherical cone metrics, but see \cite{LT} and \cite{MonP} for 
alternative proofs (the latter source sets this into a much broader context). This argument requires the construction of the 
Dirichlet polygon of a conic surface with constant curvature $K>0$ as in \cite{BLP}.
For $\theta \in (0,2\pi)$, let $\mathbb{M}^2_{K}(\theta)$ denote the spherical suspension over the circle of length $2\pi\theta$.
Now, if $(\Si,g)$ is conic with curvature $K > 0$, then for each point $p \in \Sigma$, singular or not, let $\theta_p$ denote its
cone angle. There is a star-shaped region $\mathcal{E}_{p} \subset \mathbb M^2_K(\theta_p)$ and an exponential map
$$
\exp_p : \mathcal{E}_p \rightarrow \Sigma, 
$$
which sends the `pole' $o \in \mathbb M^2_K(\theta_p)$ to $p$; these are determined as follows. Choose an isometric
identification of a neighborhood of $o$ in $\mathbb M^2_K(\theta_p)$ with a neighbourhood of $p \in \Si$, and declare that
$v \in \mathcal{E}_p$ if and only if there exists a geodesic segment $[p,x]$ in $\Sigma$ corresponding to $[o,v]$ in the sense 
that both segments have the same length and coincide via the identification above in the neighborhoods of $o$ and $p$.
We write $x = \exp_p(v)$ when $v \in \mathcal{E}_p$.   The Dirichlet polygon $\calD_p \subset \calE_p$ consists of all $v \in \calE_p$ 
such that the corresponding segment $[p,x]$ in $\Sigma$ is {\em minimizing}. It is shown in \cite{BLP} that $\calD_p$ is a polygon 
in $\mathbb{M}^2_K(\theta_p)$ and $\Sigma$ is obtained from $\calD$ by face identifications.
\begin{proposition}
If $k \geq 3$, and $\Si = S^2$ and there exists a spherical cone metric $g$ on $\Sip$ with cone angle parameters $\vec\be \in (-1,0)^k$,
then necessarily $\vec\be \in \calF$. 
\label{troyanov-suff}
\end{proposition}
\begin{proof}
For each $j\in \{1, \ldots, k\}$, consider the Dirichlet polygon $\calD_{p_j} \subset \mathbb{M}^2_K(\theta_j)$. Clearly 
\[
K \mbox{Area}(\Sip,g) = K\mbox{Area}\,\calD_{p_j} < K\mbox{Area}\, \mathbb{M}^2_K(\theta_j)=2\theta_j = 4\pi (1 + \be_j).
\]
This inequality is strict since $(\Sip,g) \neq \mathbb{M}^2_K(\theta_j)$.  On the other hand, by \eqref{eq:GB2},
\[
K \mbox{Area}(\Sip,g) = 2\pi(2+\textstyle\sum_i \beta_i),
\]
so $\sum \be_i < 2 \be_j$ for each $j$, i.e., $\vec\be \in \calF$. 
\end{proof}

We conclude this subsection by examining the geometry of $\calF$ more closely. Define
\begin{align*}
\mathsf{Hyp} &= \{ \vec\beta \in (-1,0)^k: \textstyle\sum_j \beta_j < -2\}\\ 
\mathsf{Euc} &= \{ \vec\beta \in (-1,0)^k: \textstyle\sum_j \beta_j = -2\}\\
\mathsf{Sph} &= \{ \vec\beta \in (-1,0)^k: \textstyle\sum_j \beta_j > -2\} \cap \calF. 
\end{align*}
By Theorem~\ref{fundexist}, there exists a {\em hyperbolic/Euclidean/spherical} cone metric on $\Sip$ with cone 
angle parameters $\vec\be$ if and only if $\vec\beta \in \mathsf{Hyp}/\mathsf{Euc}/\mathsf{Sph}$. Since $\calF$ is described by a set
of linear inequalities, it is the interior of a convex polyhedral cone in $\RR^k$. 
\begin{lemma}
$\mathsf{Sph}$ is the open cone over $\mathsf{Euc}$ with vertex at the origin: $\mathsf{Sph}= 
\{ \lambda \vec\beta: \vec\beta \in \mathsf{Euc}, 0 < \lambda < 1 \}$.
\end{lemma}
\begin{proof}
If $\sum_i \beta_i =-2$ and $0 < \lambda < 1$, then $\sum_i\lambda \beta_i = -2\lambda < 2\lambda \beta_j$ for each $j \in \{1, \ldots, k\}$,
i.e.~$\lambda \vec\beta \in \mathsf{Sph}$. Conversely, if $\vec\beta \in \mathsf{Sph}$, then there exists $\lambda \in (0,1)$ 
such that $\sum_i \lambda^{-1}\beta_i = -2$. It remains to check that $\lambda^{-1}\vec\beta \in (-1,0)^k$, but since 
$\vec\be \in \calF$, we have
\[
-2=\sum_{i=1}^k \lambda^{-1}\beta_i < 2 \lambda^{-1}\beta_j,
\]
as desired. 
\end{proof}
This `projection' plays a significant role in \cite{MonP}.

\subsection{Metric spaces with curvature bounds}\label{metric}
In this final part of \S 2, we briefly review basic properties of metric spaces with curvature bounds. The 
relevance is that a conic surface with all cone angles less than $2\pi$ has a lower curvature bound in 
the sense of triangle comparison. A basic reference for this material is \cite{BBI}, but see also \cite{BLP} and \cite{BGP} for more advanced treatments. 

A metric space $(X,d)$ is called a {\em length space} if the distance $d(x,y)$ between any two points $x,y \in X$ is 
equal to the infimum of lengths of continuous curves connecting $x$ and $y$ (which a priori might be infinite). 
A locally minimizing curve is called a geodesic. A geodesic segment $[x,y]$ connecting points $x,y \in X$ is
a minimizing geodesic between these points (this notation is ambiguous since there may be  more than one such 
geodesic segment), and its length is denoted $|xy|$. 
A {\em geodesic length space} is a length space where any two points $x,y \in X$ may be connected by a 
minimizing geodesic. If $(X,d)$ is a complete, locally compact length space, then the singular version of
the Hopf-Rinow theorem implies that $X$ is a geodesic length space.

For any $K_0 \in \RR$, let $\mathbb{M}^2_{K_0}$ be the $2$-dimensional model space with constant curvature
$K_0$. A complete, locally compact length space $X$ has curvature $\geq K_0$ if every point $x \in X$ has a
neighbourhood $U_x$ such that triangles with vertices in $U_x$ are `thicker' than in $\mathbb{M}^2_{K_0}$. This means 
that if $\Delta=[p,q] \cup [q,r] \cup [r,p]$ with $p,q,r \in U_x$ and $\bar{\Delta}=[\bar{p},\bar{q}] \cup [\bar{q},\bar{r}] 
\cup [\bar{r},\bar{p}]$ is a comparison triangle in  $\mathbb{M}^2_{K_0}$ with $|pq| = |\bar{p} \bar{q}|$,
$|pr| = |\bar{p}\bar{r}|$, then $|ps| \geq |\bar{p}\bar{s}|$ for any $s \in [q,r]$ where $\bar{s} \in [\bar{q},\bar{r}]$ is 
the point corresponding to $s$. If $X$ is $1$-dimensional and $K_0>0$, we require 
for consistency that $\operatorname{diam}(X) \leq \pi/\sqrt{K_0}$. We also call a geodesic length space 
space with curvature $\geq K_0$ an {\em Alexandrov space}. 

Suppose now that $Y$ is an Alexandrov space with curvature $\geq 1$ and $\operatorname{diam}(Y) \leq \pi$. 
For $K_0 \leq 0$ the $K_0$-cone over $Y$, $C_{K_0}(Y)$ is the space $\RR^+ \times Y/\{0\} \times Y$, with metric
\[
\mbox{dist}( (r_1,y_1), (r_2,y_2)) = \mbox{dist}_{\mathbb{M}^2_{K_0}}( \bar{q}_1, \bar{q}_2),
\]
where $[o,\bar{q}_1] \cup [o, \bar{q}_2] \subset \mathbb{M}^2_{K_0}$ is a `hinge' with $|o \bar{q}_j|=r_j$, $j = 1, 2$,
and (by definition) angle at $o$ equal to $\mbox{dist}_Y(y_1, y_2)$. For $K_0 >0$ the $K_0$-suspension of $Y$ is the space
$[0,\pi/\sqrt{K_0}] \times Y$ with each end $\{0\} \times Y$ and  $\{\pi/\sqrt{K_0}\} \times Y$ identified to points,
and with metric defined as before.  In both cases we obtain Alexandrov spaces with curvature $\geq K_0$. Note 
that $\mathbb{M}^2_{K_0}$ is the $K_0$-cone over $S^1$ (the circle of length $2\pi$) when $K_0 \leq 0$ and the
$K_0$-suspension of $S^1$ when $K_0>0$.  More generally, any conic surface with constant curvature $K_0$ 
and with all cone angles less than $2\pi$ are Alexandrov spaces with curvature $\geq K_0$.

Triangle comparison shows that geodesics in an Alexandrov space do not branch. For conic surfaces with constant 
curvature and cone angles less than $2\pi$, a simple geometric argument shows that any minimizing geodesic 
contains no conic points in its interior.

The following basic diameter estimate (with corresponding rigidity statement) for Alexandrov 
spaces with curvature $\geq K_0 >0$ is used in \S \ref{ss:first_eigenvalue}.
\begin{theorem}[\cite{BBI}] If $X$ is an Alexandrov space with curvature greater than or equal to $K_0 >0$, then 
$\mbox{diam}(X)\leq \pi/\sqrt{K_0}$, with equality if and only if $X$ is the $K_0$-suspension 
of an Alexandrov space with curvature $\geq 1$. 
\label{th:diam}
\end{theorem}
This implies that a conic surface $\Sigma$ with constant Gauss curvature $K_0>0$, cone angles 
less than $2\pi$ and diameter $\pi/\sqrt{K_0}$ is necessarily the $K_0$-suspension of a circle of length $\leq 2\pi$.
The underlying space is thus $S^2$ and there are exactly two conic points with equal cone angles. In particular, if
the cone angles are $2\pi$, then $\Sigma$ is the round sphere $S^2$ with curvature $K_0$. 

\section{Preliminaries from geometric analysis} 
We now describe the differential operators related to this problem.

\subsection{Curvature equations and Bianchi gauge}
The study of constant curvature metrics on surfaces customarily uses the second order nonlinear 
differential operator $g \mapsto K^g$, where $K^g$ is the Gauss curvature of the metric $g$.
However, in two dimensions, the Ricci tensor $\Ric^g$ is always pure-trace, $\Ric\,^g = K^g g$, 
so $g$ has constant curvature if and only if it is Einstein, $\Ric^g = K_0 g$ for some $K_0 \in \RR$, 
so we can use these two equations interchangeably. All constant curvature metrics near to a given one are 
solutions of the equation
\begin{equation}
(h,K) \longmapsto E^{g}(h,K) := (K^{g + h} - K) (g + h) = 0;
\label{eq:ein}
\end{equation}
where $h$ is a symmetric $2$-tensor with small norm, and $K$ is a constant near to $K_0$. 
If $K$ is not being regarded as a variable, we write simply $E^g(h)$. 
This equation is invariant under the action of the diffeomorphism group, hence cannot be
elliptic. The tangent space to the orbit through $g$ of this action consists of all symmetric 
$2$-tensors $k = (\delta^{g})^* \om$, where $\om$ is a $1$-form, or equivalently, 
$k = \frac 12 \calL_X g$ where $X$ is the vector field metrically dual to $\om$. Here $(\delta^g)^*$
is the adjoint of the divergence from symmetric $2$-tensors to $1$-forms, i.e.\ the symmetrized covariant derivative
\[
\left((\delta^{g})^* \om\right)_{ij} = \frac12 \left( \omega_{i; j} + \omega_{j; i} \right).
\]
Note that $\tr^g (\delta^g)^* = -\delta^g: \Omega^1 \to \Omega^0$ is the negative of the ordinary 
codifferential on $1$-forms.

Formally, the orthogonal complement of the tangent space to the diffeomorphism orbit at $g$ 
is the nullspace of the adjoint of $(\delta^g)^*$, i.e., of $\delta^{g}: S^2 \to \Omega^1$. 
The system $h \mapsto (E^g(h),\delta^g(h))$ is elliptic, and its solutions (with $|h|_g$ small) 
correspond to constant curvature metrics $g+h$. We use a slightly different gauge, however, 
adjoining to $E^g$ the Bianchi operator 
\begin{equation}
B^{g} = \delta^{g} + \frac12 d \, \tr^{g},
\label{eq:bianchi}
\end{equation}
because it has some convenient features. The Bianchi identities imply that $B^g(\Ric^g) = 0$ in any
dimension; specific to two dimensions, however, is the fact that $B^g$ annihilates pure trace tensors:
\begin{equation}
B^{g} (f g) = \delta^g(f g) + df = -df + df = 0
\label{eq:bianchi-trace}
\end{equation}
for all functions $f$, which implies that $B^g (K^g \, g) = 0$ too.  In any case, the range of $E^g$ lies 
in the kernel of $B^g$. 

We `roll up' the system $(E^g,B^g)$ into the single operator acting on symmetric $2$-tensors,
\begin{equation}
\begin{array}{rcl}
N^{g}(h) & = &  E^{g}(h) + (\delta^{g + h})^* B^{g}(h) \\[0.5ex]
& = &  \left( K(g+h) - K_0\right)(g+h) + (\delta^{g + h})^* B^{g}(h).
\end{array}
\label{eq:gein}
\end{equation}
Any solution $h$ to $E^g(h) = 0$, $B^g(h) = 0$ obviously satisfies $N^g(h) = 0$. The converse
is true under certain circumstances. Recall the standard Weitzenb\"ock formula 
\[
B^g (\delta^g)^* = \frac12\left((\nabla^g)^* \nabla^g - \Ric^g\right) =: P^g.
\]
Applying $B^{g+h}$ to $N^g(h) = 0$ yields $P^{g+h} B^g(h) = 0$.  Thus if we can prove that $P^{g+h}$ is invertible (or at 
least injective), we then conclude that $B^g(h) = 0$ and hence $E^g(h) = 0$, or equivalently, $K(g+h) = K_0$. 

\subsection{The linearized curvature operators}\label{ss:lin_curv}
Now consider the linearizations of the operators appearing in the last subsection.

The starting point is the formula, valid in general dimensions, for the linearization of the Einstein operator:
\begin{equation}
\left. DE^g\right|_{h=0} = \frac12 \left(\nabla^*\nabla - 2 {\circR}{}^g\right) - (\delta^g)^* B^g;
\label{eq:linein}
\end{equation}
here ${\circR}{}^g$ is the curvature tensor acting as a symmetric endomorphism on the bundle of symmetric 
$2$-tensors, see \cite{Be}. Decomposing $h$ into its tracefree and pure trace parts, $h = h^0 + f \cdot g$, then
in two dimensions, 
\[
{\circR}{}^{\, g} (h) = -K^g \, h^0 + K^g f \cdot g.
\]
Now recall the conformal Killing operator
\[
\mathcal{D}^g \omega := (\delta^g)^* \omega + \frac{1}{2} \delta^g(\omega)g:\Omega^1(\Si) \to \calC^\infty(\Si;S^2_0),
\]
which is the trace-free part of $(\delta^g)^* \omega$, and hence also the adjoint of $B^g$ restricted to trace-free tensors. 
In terms of these, 
\begin{equation}
\begin{split}
& \left. D E^g\right|_0 (h^0 + f \cdot g) = \\ & \quad \left( \frac12 (\nabla^* \nabla + 2K^g ) - \calD^g B^g \right) h^0 
+ \left( \frac12 ( \Delta^g - 2 K^g ) f + \frac12 \delta^g \delta^g h^0 \right) \cdot g, 
\end{split}
\label{eq:decompDE}
\end{equation}
hence the linearization of the Bianchi-gauged Einstein operator is
\begin{equation}
L^g:= \left. DN^g\right|_{h=0} = \frac12 \left(\nabla^*\nabla - 2 {\circR}{}^g\right). 
\label{eq:lingein}
\end{equation}
This simple expression is one motivation for introducing the Bianchi gauge.  Separating 
into tracefree and trace parts, and with $K^g \equiv K_0$, we have
\begin{equation}
L^g (h^0  + f \cdot g) = \frac12 \left( \nabla^* \nabla + 2 K_0 \right) h^0 
+ \left( \frac12 (\Delta^g - 2K_0)f \right) \cdot g.
\label{eq:geindec}
\end{equation}
The operator in the second term on the right involves the scalar Laplacian. By convention henceforth,
our scalar Laplacian is the one with nonnegative spectrum, i.e., minus the sum of second derivatives squared.

Differentiating \eqref{eq:ein} at a metric $g$ with $K^g = K_0$ gives
\[
DE^g (h) = DK^g (h)  \cdot g + (K^g - K_0) h = DK^g(h) \cdot g,
\]
so comparing this with \eqref{eq:decompDE} we obtain the two formul\ae
\begin{eqnarray}
DK^g(h^0 + f \cdot g) = \frac12 \left( (\Delta^g - 2K_0) f + \delta^g \delta^g h^0\right), \label{eq:DK} \\
\left( \frac12 ( \nabla^* \nabla + 2K_0) - \calD^g \delta^g \right) h^0 = 0. \label{eq:extraformula}
\end{eqnarray}
Note that \eqref{eq:extraformula} holds for every trace-free $h^0$, so this is a Weitzenb\"ock identity. 
On the other hand, \eqref{eq:DK} is called Lichnerowicz' formula, see \cite{Be}, \cite{Tr}.  

\medskip

There are three useful intertwining formul\ae.  First, linearizing the identity $B^{g+h}N^g(h) = P^{g+h}B^g(h)$ at $h=0$, where 
$K^g = K_0$, gives
\begin{equation}
B^g L^g = P^g B^g.
\label{eq:intertwine}
\end{equation}
Note that both sides vanish identically on pure trace tensors. Next, taking the adjoint of this equation gives
\begin{equation}
L^g \calD^g = \calD^g P^g.
\end{equation}
This will be useful in \S 5.1 below. Finally, using the most classical Weitzenb\"ock identity, 
$\Delta^g_1 = (\nabla^g)^* \nabla^g + \Ric^g$, for the Hodge Laplacian on $1$-forms, we obtain that 
\begin{equation}
P^g = \frac12\left( \Delta^g_1 - 2 K^g \right).
\label{eq:Weitz}
\end{equation}
Consequently, if $K^g = K_0$ is constant, then 
\begin{equation}
\delta^g P^g = \frac12 (\Delta^g - 2K_0) \delta^g.
\label{intertwine2}
\end{equation}

\subsection{Transverse-traceless tensors}
A key role in our analysis is played by the space of transverse-traceless tensors, both smooth or
with poles of order $1$ at $\frakp$.

Let $g$ be a smooth metric on the compact Riemann surface $\Si$, and define
\begin{equation}
\Stt = \{\kappa \in \calC^\infty(\Si; S^2(T^* \Si)): \delta^g \kappa = 0,\ \tr^g \kappa = 0\}.
\label{eq:tt}
\end{equation}
There is an identification of $\Stt$ with the tangent space at $g$ to the space of all conformal structures 
modulo diffeomorphisms on $\Si$.  This is true even in higher dimensions, but when the dimension is greater 
than $2$, $\Stt$ is infinite dimensional. In $2$ dimensions, however, $\delta^g: H^1(\Si; S^2_0) \to L^2(\Si; \Lambda^1)$ is 
elliptic, so its nullspace $\Stt$ is finite dimensional. (This ellipticity is easy to check: the bundles $S^2_0$ and $\Lambda^1$ 
both have rank $2$; furthermore, the symbol of $\delta^g$, evaluated on the covector $\xi$, is contraction with $\xi$, 
which is an isomorphism.)   In fact, $\Stt$ is canonically identified with the space of holomorphic quadratic differentials 
on $\Si$, see \cite{Tro}, hence when $\gamma > 1$, then
\begin{equation}
\dim \Stt = 6\gamma-6.
\label{eq:ttdim}
\end{equation}
The dimension is $0$ and $2$ for $\gamma = 0, 1$. 

There is an important special feature in two dimensions. 
\begin{proposition} If $\dim \Si = 2$, then $\Stt$ is conformally invariant. In other words, 
if $\tilde{g} = e^{2\phi}g$ are any two conformally related metrics on $\Si$, then
\[
\tr^{\tilde{g}}h = 0,\ \delta^{\tilde{g}}h = 0 \Longleftrightarrow \tr^{g}h = 0,\ \delta^{g}h = 0.
\]
\end{proposition}
\begin{proof} The fact that $\kappa \in \Stt$ is simultaneously trace-free with respect to
both $g$ and $\tilde{g}$ follows from the general formula
\[
\tr^{\tilde{g}} h = e^{-2\phi}\tr^g h.
\]
Next, the identity $\nabla^{\tilde{g}}_X Y = \nabla^g_X Y + d\phi(X)Y + d\phi(Y) X - g(X,Y)\nabla^g \phi$ 
yields that in $n$ dimensions, 
\[
\delta^{\tilde{g}} \kappa = e^{-2\phi}\left( \delta^g \kappa + (\tr^g \kappa) d\phi + (2-n) \iota(\nabla^g\phi) \kappa\right). 
\]
Thus if $n=2$ and $\tr^g \kappa = \delta^g \kappa = 0$, then $\delta^{\tilde{g}} \kappa = 0$.
\end{proof}
\noindent In particular, if $\olg$ is smooth and $e^{2\phi} \olg = g $ is conic, then $\Stt(\olg) = \Stt(g)$.

Now, fix $\frakp$ as before, and consider the space
\begin{equation}
\Stt^\sing = \{\kappa:  \delta^g \kappa = \tr^g \kappa = 0,\ \ |\kappa|_{\olg} = \calO(|z_j|^{-1})\  j = 1, \ldots, k\},
\label{eq:ttsing}
\end{equation}
where $z_j$ is a local holomorphic coordinate centered at $p_j$. The elements are the meromorphic quadratic differentials 
with at most simple poles at $\frakp$. By the Riemann-Roch formula, 
\begin{equation}
\dim \Stt^\sing = \dim \Stt + 2k.
\label{eq:ttsingdim}
\end{equation}
when $\gamma>1$. (The dimension is $2k-3$ for $\gamma=0$ and $2k$ for $\gamma=1$.)
It is trivial from the definition that 
\[
\kappa \in \Stt^\sing \Longrightarrow B^g \kappa = 0\ \ \mbox{on}\ \Sip.
\]
\begin{proposition} Let $(\Si,g)$ be a compact smooth surface. Then the kernel of $L^g$ on $S^2_0$ equals $\Stt$. 
\label{pr:kerLtf}
\end{proposition} 
\begin{proof}
If $\kappa \in \Stt$, then \eqref{eq:geindec} and \eqref{eq:extraformula} give that $L^g \kappa=0$. Conversely, if $L^g \kappa=0$ 
and $\tr^g\kappa = 0$, then integrating by parts, we find that
\[
\langle \calD^g\delta^g \kappa, \kappa \rangle = \| \delta^g \kappa \|^2 = 0 \Longrightarrow \kappa \in \Stt,
\]
as claimed.
\end{proof}

To conclude, we complete the description of the nullspace of $L^g$. 
\begin{proposition}
Suppose that $(\Si,g)$ has constant curvature $K_0$. Then the kernel of $L^g$ on pure trace tensors 
consists of tensors of the form $h = f \cdot g$ where $f$ lies in the eigenspace of the scalar Laplacian 
with eigenvalue $2K_0$. In particular, this kernel is trivial when $K_0<0$, it consists of the constant functions 
when $K_0=0$ and is equal to the usual three dimensional first nonzero eigenspace of the Laplacian on 
functions when $\Si = S^2$ and $K_0>0$. 
\label{pr:kerLt}
\end{proposition}

\subsection{Local deformation theory for smooth surfaces}
To warm up for the corresponding theorem in the conic setting, we now use this Bianchi-gauged Einstein 
formalism equation to reprove the (well-known) local structure theory of the Teichm\"uller space of a smooth 
compact surface $\Si$. This is only a small modification of the arguments in \cite{Tr} but we also include the cases 
where the genus $\gamma$ is $0$ or $1$, where the proofs require a bit more work. 

Let $\Si$ be a compact surface with genus $\gamma \geq 0$  and $g$ a smooth metric with constant curvature 
$K_0$ on $\Si$. We first show that the space $\calM_{K_0}$ of {\it all} metrics with constant curvature $K_0$ (without
an area normalization) is an infinite dimensional Banach manifold near $g$; the next step is to prove that the intersection 
of $\calM_{K_0}$ with a small ball in the slice $g + \Stt$ coincides with the space of nearby solutions of the Bianchi-gauged 
equation $N^g(h) = 0$, and is a smooth manifold of dimension $6\gamma-6$ when $\gamma > 1$. This intersection has 
dimensions $2$ and $0$ when $\gamma = 1$ and $0$, respectively, but is only included in the nullspace of $N^g$ then. 
This intersection parametrizes the set of constant curvature metrics near to $g$ which are not equivalent 
to $g$ by diffeomorphisms. We complete 
this picture by establishing that any metric $g' \in \calM_{K_0}$ near to $g$ can be written uniquely as $F^* (g+h)$ 
where $F$ is a diffeomorphism close to the identity and $g+h$ lies in this slice. This shows that the `flowbox' 
associated to the slice covers a full neighbourhood of $g$ in $\calM_{K_0}$.   An additional important result is that 
the identity component of the diffeomorphism group acts properly and freely on $\calM_{K_0}$ when $K_0 < 0$; 
we review this argument in \S 6 when it is generalized to the conic case. 

\begin{proposition}
The space $\calM_{K_0}$ is a Banach submanifold in the space of all $\calC^{2,\alpha}$ metrics in some neighbourhood of $g$.
\end{proposition}
\begin{proof} Following \cite{Tr}, let $\calU$ be a neighbourhood of $g$ in the space of $\calC^{2,\alpha}$ metrics, 
and observe that $\calU \ni g' \mapsto K^{g'} \in \calC^{0,\alpha}(\Si)$ is a smooth mapping. Next, by \eqref{eq:DK},
\begin{equation}
DK^g(f \cdot g) = \frac12 (\Delta^g - 2K_0)f \cdot g.
\label{restDK}
\end{equation}
If $K_0 < 0$, then $DK^g$ is surjective even when restricted to the subspace of pure trace tensors, so the implicit function theorem shows
that $\{g+h \in \calC^{2,\alpha}: |h|_g < \e, K^{g+h} = K_0\}$ is a Banach submanifold if $\e> 0$ is small enough.

If $K_0 = 0$, then the restriction of $DK^g$ to pure trace tensors is the scalar Laplacian, which has both kernel and cokernel 
identified with the space of constant functions. Furthermore, by the Gauss-Bonnet theorem, $\int_{\Si} K^{g'}\, dA_{g'} = 0$ 
for any metric $g'$. To compensate for the fact that the image satisfies a linear constraint which is determined by
the varying metric $g'$, let $\Pi_0^g$ denote the $L^2(dA_g)$ orthogonal projection from $\calC^{0,\alpha}$ onto the subspace 
$V_0^\perp$ of functions $f$ such that $\int_\Si f\, dA_g = 0$. We consider the restricted map $g' \longmapsto \Pi^g_0 (K^{g'}) $ from 
the space of all $\calC^{2,\alpha}$ metrics $g'$ with $\int_{\Si} dA_{g'} = \int_{\Si} dA_g$ to the subspace $V_0^\perp$. The intersection 
of the tangent space of the domain with the space of pure trace tensors equals $\{fg: \int_\Si f\, dA_g = 0\}$, 
and $DK^g$ is an isomorphism from this space onto $V_0^\perp$. The implicit function theorem may now be used exactly as before
to show that the set of nearby metrics $g'$ which satisfy $\Pi^g_0(K^{g'}) = 0$ is a Banach submanifold. However, clearly this equation implies that $K^{g'} = 0$ as well.

For the final case, recall that if $(\Si,g)$ is the round sphere with curvature $K_0$, then $2K_0$ is the first nonzero eigenvalue of 
$\Delta^g$, with corresponding eigenspace consisting of the restrictions of the linear coordinate functions on $\RR^3$.
The cokernel of $\Delta^g - 2K_0$ is three-dimensional and spanned by the space $V_1 := \{\iota^* dx_1, \iota^* dx_2, \iota^* dx_3\}$. 
Let $\Pi_1^g$ denote the $L^2$ orthogonal projection onto $V_1^\perp$; for simplicity we denote the range of this mapping
on $\calC^{0,\delta}$ also by $V_1^\perp$.  Very similarly to the case above, consider 
\[
\calM^{2,\delta} \ni g' \mapsto \Pi_1^g ( K^{g'}) \in V_1^\perp.
\]
The restriction of the linearization of this map to the space of pure-trace tensors at $g$ is 
$\Pi_1^g \circ \tfrac{1}{2}(\Delta^g - 2K_0)$, the image of which equals $V_1^\perp$. By the implicit function theorem,
\[
\calU \cap \calM^{2,\delta}_\cc  := \{g' = g+ h \in \calM^{2,\delta}(\Si): ||h||_{g; 2,\delta} < \epsilon,\  \Pi_1^g(K^{g'}) = K_0 \}
\]
is a Banach submanifold. The elements $g'$ in this set are precisely the metrics with $K^{g'} = K_0 + \sum_{j=1}^3 a_j x_j$. However, 
the well-known Kazdan-Warner condition states that if the function $K$ is the scalar curvature
of any metric on $S^2$, then  
\[
\int_{S^2} \nabla K \cdot \nabla x_i\, dA_g = 0, \quad i = 1, 2, 3.
\]
Hence if $K^{g'}$ has this form, then the coefficients $a_i$ all vanish, and so $K^{g'} = K_0$. 
We conclude that the space of metrics  $g'$ near to $g$ with $K^{g'} = K_0$ is a smooth submanifold
of $\calM^{2,\delta}(\Si)$. 
\end{proof}

We have now proved that if $K^g = K_0$, then 
\begin{equation}
\begin{split}
T_g \calM_{K_0} & = \{h: DK^g(h) = 0\} = \\ 
&  \{h = h^0 + f \cdot g:  \tr^g h^0 = 0, \, (\Delta^g - 2K_0)f + \delta^g \delta^g h^0  = 0\}.
\end{split}
\label{eq:tanmet}
\end{equation}
It follows from this that if $K_0 < 0$, then the trace coefficient $f$ is determined uniquely by 
the trace-free part $h^0$. If $K_0 \geq 0$, then the projection of $f$ onto the orthogonal complement
of the nullspace of $\Delta^g - 2K_0$ is determined by $h^0$. 

Equation \eqref{eq:tanmet} can be sharpened using the following decomposition.
\begin{lemma}
If $\Si$ is a smooth compact surface, and $g$ is any $\calC^{2,\alpha}$ metric on it, then 
\[
\calC^{2,\alpha}(\Si; S^2_0 T^* \Si) = \operatorname{ran} \calD^g \oplus \operatorname{ker} \delta^g.
\]
\end{lemma}
\begin{proof}
The operator $\delta \calD = P$ is elliptic and has a generalized inverse $G$ which is a pseudodifferential operator
of order $-2$ and satisfies $GP = PG = \mathrm{Id} - \Pi$, where $\Pi$ is the projector onto the finite dimensional 
nullspace of $P$ (this nullspace is trivial when $K_0 < 0$).  Then $G' = G \delta \in \Psi^{-1}$ is a left inverse for 
$\calD$, up to a finite rank error. Since $G': \calC^{1,\alpha}(S^2_0 T^* \Si) \to \calC^{2,\alpha}(T^* \Si)$ is bounded, we 
conclude that $\calD: \calC^{2,\alpha} \to \calC^{1,\alpha}$ has closed range.  

Now suppose that $h \in \calC^{2,\alpha}(S^2_0 T^* \Si)$. We seek $\omega \in \calC^{3,\alpha}(T^* \Si)$ and 
$\kappa \in \Stt$ such that $h = \calD \omega + \kappa$.  To find $\omega$, we solve
$\delta (h - \calD \omega) = 0$, or equivalently, $P \omega = \delta h$. Fortunately this is always possible
since $\delta h \in \mathrm{ran}\, P$.   Thus $\kappa := h - \calD \omega$ is both trace-free and divergence-free,
as required. 
\end{proof}

Let us apply this to the trace-free part $h^0$ of an element $h \in T_g \calM_{K_0}$. Writing
$h^0 = \calD \omega + \kappa$, then $\kappa \in \Stt$ and $(\Delta - 2K_0) f + \delta \delta \calD \omega = 0$;
using \eqref{intertwine2}, this last condition is equivalent to
\[
(\Delta - 2K_0) f + \delta P \omega = 0  \Longleftrightarrow  (\Delta - 2K_0)( f  + \frac12 \delta \omega) = 0.
\]
Thus in the case that $\Delta - 2K_0$ is invertible, then elements of $T_g \calM_{K_0}$ are determined by the
choice of any pair of elements $\omega \in \calC^{3,\alpha}(\Si; T^* \Si)$ and $\kappa \in \Stt$. When $K_0 = 0$,
then an extra constant is needed, while if $K_0 > 0$ then one must also include an element of $V_1$. 

The gauged deformation theory is now an easy consequence. 
\begin{proposition}
The intersection
\[
\calS_{g,\e} := \calM_{K_0} \cap \{ g + h: ||h||_g < \e, \ B^g h = 0 \}
\]
is transverse at $g$. It is contained in the space of all solutions $h$ to $N^g(h) = 0$ with $||h||_g < \e$ (it is equal
to it when $K_0 < 0$). In particular, metrics in $\calS_{g,\e}$ are smooth. Furthermore, $T_g \calS_{g,\e} = \Stt$ when $K_0 < 0$, $\Stt \oplus \RR$ when $K_0 = 0$, and
$\mbox{span}\,\{x_1, x_2, x_3\}$ when $K_0 > 0$. 
\end{proposition}
\begin{proof}
To check transversality, we must show that if $k \in \calC^{2,\alpha}(\Si; S^2T^* \Si)$, then 
$k = k_1 + k_2$ where $k_1 \in T_g \calM_{K_0}$ and $ B^g (k_2) = 0$.  However, recall that the kernel of $B^g$
contains $\phi g$ for every $\phi \in \calC^{2,\delta}$, so writing $k = f g + k^0$ where $k^0$ is trace-free, 
it suffices to choose $\phi$ so that $(\Delta^g - 2K_0)( f - \phi) + \delta^g \delta^g k^0 = 0$.
This is certainly possible when $K_0 < 0$, and since $\int_\Si \delta^g \delta^g k^0 = 0$, it is also
possible when $K_0 = 0$. When $K_0 > 0$, this equation can be solved provided $\delta^g \delta^g k^0$
is orthogonal to the span of the restrictions of the linear functions $\{x_1, x_2, x_3\}$.  However, this is true since
\[
\langle \delta^g \delta^g k^0 , x_j \rangle = \langle k^0,  (\delta^g)^* d x_j \rangle = 0;
\]
this last equality holds because $dx_j$ is (dual to) a conformal Killing field, so $(\delta^g)^* dx_j = -x_j g$ 
is pure trace. 

Next, if $h = f g + h^0 \in T_g \calS_{g,\e}$, then $B^g h = 0$, which means that $\delta^g h^0 = 0$, i.e., 
$h^0 \in \Stt$. In addition $(\Delta^g - 2K_0) f=0$, which leads to the three cases in the statement of the result. 

Finally, if $g + h\in \calS_{g,\e}$, then $g+h$ has constant curvature $K_0$ by definition, and clearly $N^g(h) = 0$. 
Conversely, if $||h|| < \e$ and $N^g(h) = 0$, then by the argument at the very end of \S 3.1, if $K_0 < 0$ then 
$B^g(h) = 0$ and $g+h$ has constant curvature $K_0$. Finally, since $N^g$ is an elliptic, quasilinear operator, $N^g(h)=0$ implies that $h$ is smooth.
\end{proof} 

To complete this picture, we show that $\calS_{g,\e}$ is a slice for the diffeomorphism action.
\begin{proposition}
If $\e$ is sufficiently small and if $\gamma > 1$, then in some neighbourhood $\calU$ of $\mathrm{Id} \in 
\operatorname{Diff}^{3,\alpha}(\Si)$, the map
\begin{equation}
\calU \times \calS_{g,\e} \ni (F, h) \longrightarrow F^*(g + h)
\label{eq:flowboxdef}
\end{equation}
is a local diffeomorphism onto a neighbourhood of $g$ in $\calM_{K_0}$. 
\end{proposition}
\begin{proof}
Every element of $\calS_{g,\e}$ is $\calC^\infty$, which implies that \eqref{eq:flowboxdef} is a $\calC^\infty$ 
mapping.  Thus it suffices to check that its linearization 
\[
\calC^{3,\alpha}(\Si,T^*\Si) \times T_g \calS_{g,\e} \ni (\om,\kappa) \longmapsto (\delta^g)^* \om + \kappa 
\in T_g \calM_{K_0}
\]
is an isomorphism. However, we have already shown how any $h = fg + h^0 \in T_g \calM_{K_0}$ 
uniquely decomposes as $ (-\frac12 \delta (\omega) g + \calD \omega) + \kappa$, as required.
\end{proof}
The action of the entire diffeomorphism group is proper and the action of the identity component of the diffeomorphism group 
is also free. This is proved in \cite[Section 2.3]{Tro}, but is discussed further and generalized at the end of \S 6 below.

\section{Elliptic operators and conic singularities}\label{s:elliptic_theory}
To generalize the results in the last section to surfaces with conic singularities, we take a brief detour to describe 
linear elliptic theory of conic operators. This is a well developed subject, and much of what is needed here could be carried 
out `by hand', but we rely on a more systematic approach using the calculus of $b$-pseudodifferential operators,
and quote the results we need from \cite{Ma}. 

Since we need to apply these results to several Laplace-type operators, we state results from the general theory
for a `generic' operator of the form
\[
A = \nabla^* \nabla  + B,
\]
acting on sections of a Hermitian vector bundle $E$ over a conic surface $(\Sip,g)$. Here $\nabla$ is a 
Hermitian connection and $B$ is a Hermitian bundle map.  As usual, the model conic metric is 
$g_\be = dr^2 + (1+\be)^2 r^2 dy^2$, and we assume that $\lim_{r \to 0} r^2 B = \tilde{B}$.  Under natural
geometric hypotheses, $\nabla$ is a conic operator, i.e., there exists a smooth basis of section $\sigma_j$ 
for $E$ such that $\nabla_{r\del_r}\sigma_j$ and $\nabla_{\del_y}\sigma_j$ are smooth (as functions of $(r,y)$) 
linear combinations of the $\sigma_i$. In the cases we consider, $E$ is a subbundle of some tensor bundle 
and $\nabla$ is the Levi-Civita connection, and this condition is easily checked to be satisfied for sections 
which are linear combinations of tensor products of the basic vector fields and $1$-forms
\[
\del_r, \frac{1}{r}\del_y, \ \ \mbox{and}\ \  dr,\, r\,dy.
\]
In a local trivialization of this type 
\begin{equation}
A = -\del_r^2 - \frac{1}{r}\del_r + \frac{1}{r^2}\tilde{A} + Q 
\label{typicalop}
\end{equation}
where $\tilde{A}$ is the $r$-independent `tangential operator', acting on sections of $\tilde{E}$,
the restriction of $E$ to the $S^1$ cross-section, and $Q$ is a lower order error in the sense that 
it vanishes in any limit of homothetic rescalings $(r,y) \mapsto (\lambda r, y)$, $\lambda \to \infty$. 
We denote by $A_0$ this operator with the remainder term $Q$ omitted, so that $A_0$ is exactly dilational
invariant. 
 
\subsection{Function spaces}
We now introduce the weighted  $b$-Sobolev and $b$-H\"older spaces, which are based on differentiations with respect
to the $b$ vector fields $r\del_r$,  $\del_y$.  Thus for $m \in {\mathbb N}$, define 
\[ 
H^m_b(\Sip,dA_g) = \{u: (r\del_r)^j\del_y^\ell u \in L^2(\Sip,dA_g),\ j+\ell \leq m \}, 
\]
and for any $\mu \in \RR$, 
\[
r^\mu H^m_b(\Sip,dA_g) = \{u = r^\mu \tilde{u}: \tilde{u} \in H^m_b(\Sip,dA_g)\}.
\]  
The fact that the normal derivative is with respect to $r\del_r$ rather than $\del_r$ means, for example, that if $\gamma > -1$ 
then $r^\gamma \in H^m_b(\Sip,dA_g)$ for every $m \geq 0$, whereas $r^\gamma$ lies in the ordinary Sobolev space
$H^m(\Sip,dA_g)$ only if $\gamma -m > -1$.

Next define $\calC^{0,\al}_b(\Sip)$ to equal the space of functions on $\Sip$ in $\calC^{0,\al}(\Sip)$ such that near each $p_j$,
\[ 
\sup_{0 < R < R_0} \sup_{R \leq r,r' \leq 2R} 
\frac{|u(r,y) - u(r',y')|R^\al}{|(r,y)-(r',y')|^\al} \leq C; 
\] 
$\calC^{m,\al}_b(\Sip)$ consists of all functions $u$ such that (near each $p_j$),  $(r\del_r)^j \del_y^\ell u \in
\calC^{0,\al}_b$. Finally, $r^\mu \calC^{m,\al}_b(\Sip) = \{u = r^\mu v: v \in \calC^{m,\al}_b(\Sip)\}$. 

All of these definitions continue to make sense for operators acting between sections of bundles. For notational
convenience, we describe the general results below only for scalar operators, but everything works directly for systems as well.

\subsection{Polyhomogeneity and indicial roots}
The most natural replacement for smooth functions in this setting involves the notion of polyhomogeneity. 
A function $u$ is polyhomogeneous with index set $I = \{\gamma_i,N_i\} \subset \CC \times \NN$ if $u \in
\calC^\infty(\Sip) $ and near each $p_j$, 
\[ 
u \sim \sum_{i} \sum_{\ell = 0}^{N_i} u_{i,\ell}(y) r^{\gamma_i} (\log r)^\ell.
\] 
An index set is a countable collection of pairs $(\gamma_i, N_i)$ such that $\mathrm{Re}\, \gamma_i \to \infty$ 
(or equivalently, there are only finitely many $\gamma_i$ in each left half-plane $\mbox{Re}\, \gamma < C$). 
This is convenient notation for keeping track of which exponents appear in this expansion. 

To see how such functions arise naturally in the present setting, consider the simple ``monomial'' solutions of 
the model operator $A_0$. 
\begin{definition} A number $\gamma \in \CC$ is called an indicial root for $A$ at $p_j$
of multiplicity $N$ if there exists some $\phi \in \calC^\infty(S^1;\tilde{E})$ such that 
$A(r^\gamma (\log r)^N \phi(y)) = \calO(r^{\gamma - 1}(\log r)^N)$ (as opposed to the expected 
singularity $\calO(r^{\gamma-2}(\log r)^N)$), but that the same is not true if $N$ is replaced by $N+1$.

Denote by $\Gamma(A)$ the set of all indicial roots of $A$ (over all $p_j$) and by 
$\tilde{\Gamma}(A)$ the set $\{(\gamma,N): \gamma \in \Gamma(A),\ N = 
\mbox{logarithmic multiplicity at}\ \gamma\}$.
\end{definition} 
It follows directly from this definition that if $A$ is as in \eqref{typicalop}, then $\gamma \in \Gamma(A)$ if and only 
if there exists a section $\phi$ of $\tilde{E}$ such that 
\[
(\tilde{A} - \gamma^2)\phi = 0,
\]
or in other words, $\gamma^2 \in \mbox{spec}(\tilde{A})$ and $\phi$ is a 
corresponding eigensection. In particular, $\gamma \in \Gamma(A)$ if and only
if $-\gamma \in \Gamma(A)$. (This symmetry is due, of course, to the symmetry of
$A$ on $L^2$.) 

To simplify notation, we shall assume that all indicial roots are real; this 
is the case for the specific operators considered below. 

\subsection{Mapping properties}
We assume that our operators have polyhomogeneous coefficients.  For any $\mu \in \RR$, $m \in \NN$, it is straightforward that 
\begin{equation}
\begin{array}{rcl}
A: r^{\mu}H^{m+2}_b(\Sip;E) &\longrightarrow & r^{\mu-2}H^{m}_b(\Sip;E); \\
A: r^{\mu}\calC^{m+2,\al}_b(\Sip;E) &\longrightarrow &  r^{\mu-2}\calC^{m,\al}_b(\Sip;E)
\end{array} 
\label{eq:mash}
\end{equation}
are bounded mappings.  
\begin{proposition}[\cite{Ma} Theorem 4.4] The mappings \eqref{eq:mash} are Fredholm if and only if 
$\mu \notin \Gamma(A)$. 
\label{pr:fredsh} 
\end{proposition}

Using the symmetry of $A$ with respect to the area form of $g$, it is straightforward to
check that if $\mu \notin \Gamma(A)$, then for any $m,m' \in \NN$, the cokernel of the mapping
\[
A: r^{\mu+1}H^{m+2}_b(\Sip;E) \rightarrow r^{\mu-1}H^{m}_b(\Sip;E)
\]
is naturally identified with the kernel of
\[
A: r^{-\mu+1}H^{m'+2}_b(\Sip;E) \rightarrow r^{-\mu-1}H^{m'}_b(\Sip;E).
\]
Even though we cannot directly argue using duality as with $b$-Sobolev spaces, a similar
statement is true in the H\"older setting. 
\begin{proposition} Suppose that $\mu$ is not an indicial root of $A$. Let $K_{-\mu}$ denote
the nullspace of $A$ on $r^{-\mu}\calC^{2,\al}_b(\Sip)$. If $f \in r^{\mu-2}\calC^{0,\alpha}_b$, 
then there exists an element $h \in K_{-\mu}$ such that $Au = f - h$ for some $u \in r^{\mu}\calC^{2,\al}_b$. 
In particular, if $K_{-\mu} = \{0\}$, then 
\[
 A: r^{\mu}\calC^{2,\al}_b(\Sip) \rightarrow r^{\mu-2}\calC^{0,\al}_b(\Sip).
\]
is surjective.
\label{pr:duality}
\end{proposition}
The statement of this result assumes that the dimension is $2$, but note that this result has analogues in all dimensions.
The proof proceeds by showing that the generalized inverse $G$ of $A$ acting on a particular weighted $b$-Sobolev space 
is bounded on a related $b$-H\"older space, 
see \cite{Ma} for details. 

\begin{proposition}[\cite{Ma} Corollary 4.19]
If $Au = f$ and if both the conic metric $g$ and the function $f$ are polyhomogeneous, then $u$ is polyhomogeneous. 
\label{pr:reg}
\end{proposition} 
Observe that once we {\it know} that $u$ is polyhomogeneous, then substituting an arbitrary expansion $u \sim 
\sum r^{\gamma_j} u_j(y)$ into the equation $Au = f$ and matching terms shows that each $\gamma_j$ is 
an exponent in the expansion of $f$ (shifted by $2$), or an indicial root of $A$, 
or else of the form $\zeta_j + e$ where $\zeta_j$ is an indicial root of $A$ and $e$ is an exponent 
arising in the expansion of one of the coefficients of $A$.

\begin{proposition} If $A u = f$, where $u \in r^{-\mu} \calC^{m+2,\delta}_b$ and $f \in r^{\mu-2} \calC^{m,\delta}_b$, then 
near each conic point, 
\[
u = \sum u_{j\ell}(y) r^{\gamma_j} (\log r)^\ell + \tilde{u}
\]
where the sum ranges over those elements $(\gamma_j,\ell) \in \tilde{\Gamma}(A)$ with $\gamma_j \in (-\mu, \mu)$ 
and where $\tilde{u} \in r^{\mu} \calC^{m+2,\alpha}_b$. 
\label{pr:impreg}
\end{proposition}
The intuition is that since $f$ lies in a slightly better space than expected, then $u$ decays as fast as one
might predict from the decay of $f$ except for some finite dimensional component coming from the indicial roots of $A$. 
Note that the finite sum in this expression is not necessarily in the global nullspace of $A$, but is 
formally annihilated by $A$.

\section{Applications of conic elliptic theory}
We now apply the general theory reviewed in \S 3 to the present setting.  We first compute the indicial data for
the operators $\Delta^g$, $P^g$ and $L^g$ and then discuss some implications of the information we get this way.

We shall assume for most of this section that $g$ is a conic metric which is polyhomogeneous at each 
$p_j$. The issue of metrics with finite regularity is discussed at the end of the section. 

\subsection{Indicial roots}\label{ss:indicial_roots}
If $g \in \CM(\Sip, \vec\be)$ and $p_j$ is any one of the cone points, then by definition $g$ is asymptotic to the 
model metric $g_{\be_j}$ at that point. Dropping the subscript $j$, it is therefore the case that the indicial operator
for $\Delta^g$, $P^g$ and $L^g$ at that point agrees with corresponding operator for $g_\be$, and since this
metric is flat, these indicial operators are equal 
(up to a constant factor) to the rough Laplacians $\nabla^* \nabla$ acting on scalar functions, sections of $T^*\Sip$ 
and sections of $S^2T^*\Sip$, respectively.   If $E$ is any Hermitian vector bundle with metric compatible connection 
over $(C(S^1), g_\beta)$, then  
\[
\nabla^*\nabla = r^{-2} \left( -(\nabla_{r \partial_r})^2 - (1+\beta)^{-2}(\nabla_{\partial_{y}})^2\right) \,.
\]
The calculations and results below will be carried out using the eigenfunction basis $\psi_k(y) = e^{i k y}$, $k \in \ZZ$, for
$-\del_y^2$, with corresponding eigenvalues $k^2$. 

\subsubsection*{Indicial roots of $\Delta$}
Restricting the scalar Laplacian,
\[
\Delta = r^{-2}\left(-(r\del_r)^2 - (1+\beta)^{-2}\del_y^2\right)
\]
to the $k^{\mathrm{th}}$ eigenspace, we see that $r^\zeta \psi_{\pm k}(y) \in \ker(\Delta)$ precisely when 
$\zeta =  \zeta_k^0 = k/(1+\be)$. (We allow $k$ to be both positive and negative.) Hence the 
set of indicial roots for $\Delta$ equals
\[
\Gamma_0(\be) := \left\{\frac{k}{1+\beta}: k \in \ZZ \right\}.
\]
Note that $0$ is a double root, so both $r^0 = 1$ and $\log r$ are solutions. 

\subsubsection*{Indicial roots of $P$}
The Levi-Civita connection for the metric $g_\be$ on $T^*C(S^1)$ is determined by 
\[
\nabla_{\partial_r}dr = 0, \  \nabla_{\partial_{y}} dr = r (1+\beta)^2 dy, \ 
\nabla_{\partial_r}dy = -\frac{1}{r}dy, \ \nabla_{\partial_{y}}dy =  -\frac{1}{r}dr \,.
\]
The action of $P$ on the $1$-form 
\begin{equation}
\eta = \eta_1 dr + \eta_2 (1+\beta) rdy
\label{decomp1form}
\end{equation}
is thus
\[
P \begin{bmatrix} \eta_1 \\ \eta_2 \end{bmatrix}
= r^{-2} \left(-(r \del_r)^2+1 + \begin{bmatrix}
-\del_y^2/(1+\beta)^2 & 2\del_y/(1+\beta) \\
-2\del_y/(1+\beta) & -\del_y^2(1+\beta)
\end{bmatrix} \right) \begin{bmatrix} \eta_1 \\ \eta_2 \end{bmatrix}\,.
\]
The indicial roots are calculated as for the scalar case. The equation
\[
P \left(r^\zeta \psi_k(y) \begin{bmatrix} b_1 \\ b_2 \end{bmatrix}\right) = r^{\zeta-2}\psi_k
\left ( - \zeta^2 + 1 + \begin{bmatrix}  k_\beta^2 & 2ik_\beta \\ 
-2ik_\beta & k^2_\beta \end{bmatrix}  \right) \begin{bmatrix} b_1 \\ b_2 \end{bmatrix} = \begin{bmatrix} 0 \\ 0 \end{bmatrix},
\]
where $k_\beta:=k/(1+\beta)$, has nontrivial solutions if and only if 
\[
\zeta = \zeta_{\pm,k}^1 \in \Gamma_1(\be) := \left\{ \pm 1 + \frac{k}{1+\be},\  k \in \ZZ \right\}; 
\]
the solutions corresponding to $\zeta_{\pm, k}^1$ are given by
\[
\begin{bmatrix} \eta_1 \\ \eta_2 \end{bmatrix} = 
r^{\pm 1 + \frac{k}{1+\be}} \left( \alpha_1 \psi_{k}(y) \begin{bmatrix} 1 \\ \mp i \end{bmatrix}+
\alpha_2 \psi_{-k}(y)\begin{bmatrix} 1 \\ \pm i \end{bmatrix} \right),
\]
or as real vectors
\[
\begin{bmatrix} \eta_1 \\ \eta_2 \end{bmatrix} = 
r^{\pm 1 + \frac{k}{1+\be}} \left( \alpha_1' \begin{bmatrix} \cos (k y)  \\ 
\pm \sin ( k y) \end{bmatrix}+
\alpha_2' \begin{bmatrix} -\sin (k y ) \\ \ \, \pm\cos (k y ) 
\end{bmatrix}\right).
\]

\subsubsection*{Indicial roots of $L$}
Identify the symmetric two-tensor $h$ with the triplet $(\phi_1, \phi_2, \phi_3)$ by
\begin{equation}
h = \phi_1 g_\beta + \phi_2 (dr^2 - (1+\be)^2 r^2 dy^2) + \phi_3 (1+\beta)r(dr \otimes dy + dy \otimes dr). 
\label{decomph}
\end{equation}

The first component, which is the pure trace part, immediately decouples: 
\[
\nabla^*\nabla (\phi_1 g_\be) = (\Delta \phi_1) g_\beta,
\]
so the set of indicial roots for this part is once again $\Gamma_0(\be)$.

A further calculation produces the matrix expression
\[
\nabla^* \nabla \begin{bmatrix} \phi_2 \\ \phi_3 \end{bmatrix}
= r^{-2}\left( -(r\del_r)^2+4+
\begin{bmatrix}
  -\del_y^2/(1+\beta)^2 & 4 \del_y/(1+\beta) \\
 -4\del_y/(1+\beta) &  -\del_y^2/(1+\beta)^2 
\end{bmatrix} \right)\begin{bmatrix} \phi_2 \\ \phi_3 \end{bmatrix}, 
\]
so proceeding as before,
\[
\nabla^* \nabla \left(r^\zeta \psi_k(y) \begin{bmatrix} b_2 \\ b_3 \end{bmatrix}\right) = \left( -\zeta^2 +4 +
\begin{bmatrix}
 k_\beta^2  & 4ik_\beta \\
-4ik_\beta & + k_\beta^2
\end{bmatrix} \right)
\begin{bmatrix} b_2 \\ b_3 \end{bmatrix} =
\begin{bmatrix} 0 \\ 0 \end{bmatrix},
\]
where $k_\beta:=k/(1+\beta)$, has nontrivial solutions if and only if 
\[
\zeta = \zeta_{\pm,k}^2 \in \Gamma_2(\be) := \left\{ \pm 2 +\frac{k}{1+\be} : k \in \ZZ \right\};
\]
the solutions corresponding to $\zeta_{\pm,k}^2$ are
\[
\begin{bmatrix} \phi_2 \\ \phi_3 \end{bmatrix} = 
r^{\pm 2 + \frac{k}{1+\beta}} \left( \alpha_2' \begin{bmatrix} \cos (k y) \\
\pm \sin ( k y) \end{bmatrix}+
\alpha_3' \begin{bmatrix} -\sin ( k y)  \\ \ \, \pm \cos ( k y) \end{bmatrix}
\right).
\]
\subsubsection*{The relationship between the indicial roots of $L$ and $P$}
The adjoint of the intertwining formula \eqref{eq:intertwine} for $g\in \CM_\cc$ is 
\begin{equation}
L^g \mathcal{D}^g = \mathcal{D}^gP^g.
\label{intertwine3}
\end{equation}
This formula and \eqref{eq:intertwine} imply some useful relationships between the indicial data of $P^g$ and $L^g$. 
Let us say that $(\zeta, \Phi(y))$ is an indicial pair for a second order conic elliptic operator $A$ if $A_0(r^\zeta \Phi(y)) = 0$,
where $A_0$ is the indicial operator for $A$. Thus $\zeta$ is an indicial root for $A$ and $\Phi$ is the corresponding 
coefficient section. 

Now suppose that $D$ is a first order conic operator, with indicial operator
\[
D_0 = r^{-1} \left( r \del_r \, T_D + \tilde D \right),
\]
where $T_D$ is an endomorphism (typically induced by Clifford multiplication). Clearly 
\[
D_0 (r^\zeta \Phi(y)) =  r^{\zeta-1}(\zeta \, T_D \Phi+\tilde{D}\Phi)(y). 
\]

\begin{lemma}
The operators $B$ and $\calD$ intertwine the indicial pairs for $L$ and $P$. Namely, if $(\zeta,\Phi(y))$ is an indicial pair 
for $L$, then $(\zeta-1, (\zeta T_B \Phi+\tilde B\Phi)(y))$ is an indicial pair for $P$. Similarly, if $(\zeta,\Phi(y))$ is an indicial pair 
for $P$, then $(\zeta-1, (\zeta T_\calD \Phi+\tilde{\calD}\Phi)(y))$ is an indicial pair for $L$.  
\end{lemma}
\begin{proof}
If $L(r^\zeta \Phi) = 0$, then $0=B L (r^\zeta \Phi)  = P B (r^\zeta \Phi) = P( r^{\zeta-1} (\zeta T_B \Phi+\tilde B\Phi))$. 
Similarly, if conversely $P(r^\zeta \Phi) = 0$, then $0= \calD \calP (r^\zeta \Phi)  = L\calD(r^{\zeta}\Phi) = 
L(r^{\zeta-1} (\zeta T_{\calD} \Phi+\tilde \calD\Phi))$. 
\end{proof}

One case is of particular interest.  Define
\[
X_\beta = \{\Phi:  P (r^{-1 + \frac{1}{1+\be}} \Phi) = 0\}, \qquad Y_\beta = \{ \Psi: L ( r^{-2 + \frac{1}{1+\be}}\Psi) = 0 \}.
\]
By the preceding lemma, $\calD$ maps $X_\beta$ into $Y_\beta$. 
\begin{lemma}\label{lem:iso}
The map $X_\beta \longrightarrow Y_\beta$ induced by $\calD$ as above is an isomorphism.
\end{lemma}
\begin{proof}
Dropping the subscript $0$, we first compute $\delta$ and $\delta^*$ on the model space $(C(S^1), g_\be)$. 

In terms of the bases used in the decompositions \eqref{decomp1form} and  \eqref{decomph}, we have
\[
\delta h  = \frac{1}{r} \begin{bmatrix} -r\del_r &  -r\del_r - 2  & -\del_y/(1+\be)  \\
-\del_y/(1+\be) & \del_y/(1+\be) & -r\del_r - 2 \end{bmatrix}  \begin{bmatrix} \phi_1 \\ \phi_2 \\ \phi_3 \end{bmatrix},
\]
and similarly
\[
\delta^* \eta = \frac{1}{r} \begin{bmatrix} r\del_r + 1 & \del_y/(1+\be) \\ r\del_r - 1 & -\del_y/(1+\be) \\
\del_y/(1+\be) & r\del_r - 1 \end{bmatrix}  \begin{bmatrix}  \eta_1 \\ \eta_2 \end{bmatrix}.
\]
Hence inserting 
\[
\begin{bmatrix} \eta_1 \\ \eta_2 \end{bmatrix} = r^{-1  + \frac{1}{1+\beta}}\left( a_1 \begin{bmatrix} \cos y \\ - \sin y  \end{bmatrix} 
+ a_2 \begin{bmatrix} -  \sin y \\ - \cos y \end{bmatrix}\right), 
\]
we get
\[
\delta^* \begin{bmatrix} \eta_1 \\ \eta_2 \end{bmatrix} = -2 r^{-2+\tfrac{1}{1+\beta}}
\left( a_1 \begin{bmatrix} 0 \\    \cos y  \\  -\sin y \end{bmatrix}  
+ a_2 \begin{bmatrix}0 \\   -\sin y  \\  -\cos y  \end{bmatrix} \right). 
\]
The top row, which corresponds to the coefficient of $g_\be$, vanishes, which means that $\delta^*\eta$ is traceless,
i.e., $\delta^*\eta = \calD \eta$, for any $\eta \in X_\be$.   The projection of this to the second and third rows is 
clearly invertible, as claimed. 
\end{proof}

\subsection{Geometric realization}
It is important that the indicial roots of these operators close to $0$ have geometric interpretations. 
This discussion is local, and it suffices to work with the flat model metric $g_\be$ in the unit disk. 
Set $|z| = \rho$ and $r = \rho^{1+\beta}/(1 + \beta)$.  
\begin{itemize}
\item[i)] The value $0$ is a double indicial root for the scalar operator $\Delta$,  and the solutions corresponding to it are
$r^0 = 1$ and $\log r$. Since $L ( f g_\be) = \frac12 (\Delta f )g_\be$, the functions $1$ and $\log r$ arise 
from infinitesimal `pure trace' deformations. Indeed, setting $g(a,\eta) = a r^{2\eta} g_\be$, then 
\[
\del_a g|_{a=1, \eta = 0} = g_\be, \quad \del_\eta g|_{a=1, \eta = 0} = 2 \log r \, g_\be,
\]
i.e., these two solutions correspond to scaling the metric and changing the cone angle, respectively.
\item[ii)] 
The operator $P$ has a two-dimensional family of solutions at the indicial root $1$, spanned by $r dr$ and $r^2 dy$, cf.~the 
calculations in \S \ref{ss:indicial_roots}. The vector fields dual to these $1$-forms are 
$$
X_1 = r \partial_r \quad\text{and}\quad X_2 = \partial_y \, ,
$$
which generate conformal dilations and rotations.  Note that $\rho \partial_{\rho} = (\beta+1) r \partial_r = (\beta+1) X_1$,
so these are in fact a local basis for the space of vector fields smooth across the origin which vanish at the puncture. 
\item[iii)] A symmetric $2$-tensor $h$ which is smooth across $z=0$ has norm
\[
|h| \leq \rho^{-2\be} = c r^{-\frac{2\be}{1+\be}} = c r^{-2 + \frac{2}{1+\be}}.
\]
The space of trace-free solutions corresponding to this indicial root is spanned by $r^{-2+\frac{2}{1+\be}}(dr^2 - (1+\be)^2 r^2 dy^2)$ 
and $r^{-2 + \frac{2}{1+\be}} rdrdy$. 
\item[iv)] Finally, if $\tr h = 0$ and the coefficients of $h$ blow up no faster than $|z|^{-1}$, then
\[
|h| \approx \rho^{-2\be-1} = c r^{-\frac{2\be+1}{1+\be}} = c r^{-2 + \frac{1}{1+\be}}.
\]
This matches the growth rate for $h \in \Stt^\sing \setminus \Stt$. 
\end{itemize}

\subsection{Regularization of $\Stt^\sing$} 
Since singular TT tensors satisfy $|\kappa| \sim r^{-2 + \frac{1}{1+\be}}$, a {\it regularization} map is needed to incorporate
them into the analysis below. Thus we define 
\begin{equation}
R: \Stt^\sing \longrightarrow r^\nu \calC^{m,\delta}_b(\Sip; S^2_0),
\end{equation}
for some $0 < \nu \ll 1$. When $-1 < \be < -\frac 12$, we have $-2 + 1/(1+\be) > 0$, and hence
we can choose any $\nu \in (0, -2 + 1/(1+\be))$ and define $R(\kappa) = \kappa$.  

On the other hand, suppose $-1/2 \leq \be < 0$; Lemma \ref{lem:iso} asserts that for each $j$ 
there is a unique $\omega_0^{(j)} r^{-1+\frac{1}{1+\be}}\in X_\be$ so that near $p_j$, 
\[
\kappa - \calD^g (\chi_j \omega_0^{(j)} r^{-1+\frac{1}{1+\be}}) \in r^\nu \calC^{\ell,\delta}_b.
\]
This leads us to define
\begin{equation}
R(\kappa) = \kappa - \calD^g (\sum_j \chi_j \omega_{0}^{(j)} r^{-1+\frac{1}{1+\be}}).
\label{regmap}
\end{equation}
(the sum is only over those cone points where $\be_j \geq -1/2$). 
\begin{proposition}
The map $R$ is injective. 
\end{proposition}
\begin{proof} This is obvious when $\be < -1/2$. In the other case, if $R(\kappa)=0$, then 
$\kappa = \calD^g \omega$ for some $\omega \in r^{-1+\frac{1}{1+\beta}}\calC^{m,\delta}_b$. Applying $B^g$ shows that
$P^g \omega=0$.  However, by Corollary~\ref{cor:invP}, since $-1 + \frac{1}{1+\be} > 0$, $P^g$ is injective 
on $r^{-1+\frac{1}{1+\be}}\calC^{2,\delta}_b$, so $\omega= \kappa = 0$. 
\end{proof}

\subsection{The Friedrichs extension and an eigenvalue estimate}\label{ss:first_eigenvalue}
We now return to the scalar Laplacian $\Delta^g$ and consider it as an unbounded operator acting on $L^2(\Sip)$. 
The maximal domain $\calD_{\max}(\Delta^g)$ of this mapping is by definition the set of all $u \in L^2$ such 
that $\Delta^g u \in L^2$. It is well-known that 
\begin{equation}
\Delta^g:  \calD_{\max}(\Delta^g) \longrightarrow L^2(\Sip)
\label{maxop}
\end{equation}
is a closed operator. By Proposition~\ref{pr:impreg}, there is a precise characterization of elements
in this maximal domain: 
\[
\calD_{\max}(\Delta^g) = \{ u = a_0 + \tilde{a}_0 \log r + \tilde{u}, \quad a_0, \tilde{a}_0 \in \RR,\ \tilde{u} \in r^2 H^2_b(\Sip) \}.
\]
To see how this follows from that result, note that $0$ is the only indicial root $\gamma$ in the range $(-1,1]$; this coincides
with the set of values where $r^\gamma \in L^2$ but $r^{\gamma-2} \notin L^2$. 

Since the mapping \eqref{maxop} is not self-adjoint, we cannot talk about its spectrum until we specify and restrict to
a domain of self-adjointness. We use the canonical choice of the Friedrichs domain $\calD_{\mathrm{Fr}}(\Delta^g)$, 
which consists of all functions $u \in L^2$ such that both $\nabla u$ and $\Delta u$ also lie in $L^2$.  This contains
those functions $u \in \calD_{\max}(\Delta^g)$ such that the coefficient $\tilde{a}_0$ in the expansion above 
vanishes. Note that a function $u$ lies in $\calD_{\mathrm{Fr}}(\Delta^g)$ if
$\Delta^g u = f \in L^2$ and if $u$ is the unique bounded solution of $\Delta^g u = f \in L^2$. 

It will be convenient for us to define a H\"older space analogue of this Friedrichs domain:
\begin{definition}
$\calD^{\ell,\delta}_{\mathrm{Fr}} = \{u \in \calC^{\ell,\delta}_b(\Sip):  \Delta^g u \in \calC^{\ell,\delta}_b(\Sip) \}.$
\label{HF}
\end{definition}
By Proposition \ref{pr:impreg}, we have that $u \in \calD^{\ell,\delta}_{\mathrm{Fr}}$ if and only if
\[
u = a_0 + a_1 r + a_2 r^{\frac{1}{1+\be}} + \tilde{u}, \qquad \tilde{u} \in r^2 \calC^{\ell+2,\delta}_b.
\]
Note the inclusion of the extra term $r^{\frac{1}{1+\be}}$ (which is unnecessary if $\be \leq -1/2$).

Since $(\Delta^g, \calD_{\mathrm{Fr}})$ is self-adjoint, we may talk about its spectrum.
Our final result in this section is an estimate (and rigidity statement) for its first nonzero eigenvalue.

\begin{proposition} Suppose that $(\Sip, g)$ is conic with all cone angles in $(0,2\pi)$ and has constant Gauss curvature $K_0>0$. 
Let $\lambda_1$ be the first nonzero eigenvalue of the Friedrichs extension of $\Delta^g$. Then $\lambda_1 \geq 2K_0$, 
with equality if and only if $\Sigma$ is either the round $2$-sphere with constant Gauss curvature $K_0$ or else 
there are precisely two conic points with cone angles $\theta_1=\theta_2$ and $\Sip$ is the $K_0$-suspension of a 
circle of length $2\pi (1+\be)$. 
\label{eigenbound}
\end{proposition}
\begin{proof} If $u \in \calD_{\mathrm{Fr}}(\Delta^g)$ and $\Delta^g u = \lambda_1 u$, then near each $p_j$, 
\[
u \sim a_0 + r^{\frac{1}{1+\be_j}}(a_1^+ \cos y + a_1^- \sin y) + \ldots.
\]
Consequently $|du|_g = \calO(r^{-1+ \frac{1}{1+\be_j} })$, which allows one to justify the integration by parts in
\[
\int_{\Si}  \langle \Delta_1 du, du \rangle = \int_\Si \left(|\nabla du|^2 + K_0|du|^2\right),
\]
where $\Delta_1 = \nabla^*\nabla + \Ric^g = \nabla^*  \nabla + K_0$ is the Hodge Laplacian on $1$-forms. 

Noting that $\Delta_1 du = d \Delta_0 u = \lambda_1 du$, and applying the Cauchy-Schwarz inequality,
$|\nabla du|^2 \geq \frac{1}{2} (\Delta^g u)^2$, we see that
\begin{multline*}
\lambda_1 \|du\|^2 = \int_{\Si} \langle \Delta_1 du,du \rangle  \\
\geq \int_{\Si} \frac{1}{2} (\Delta^g u)^2 + K_0 |du|^2 =  \int_{\Si}\frac{1}{2} \langle \Delta_1 du,du) +K_0 |du|^2.
\end{multline*}
The integration by parts $\langle \delta du, \delta du \rangle = \langle d\delta du, du \rangle$ is justified as before.
Rearranging this yields $\lambda_1 \geq 2K_0$, as desired.

Now consider the rigidity statement. If $\lambda_1=2K_0$, then we have equality
$|\nabla du|^2 = \frac{1}{2} (\Delta^g u)^2$, so the Hessian of $u$ is pure trace:
\[
\nabla du = -\frac{1}{2}(\Delta^g u) g.
\]
We can solve $\nabla du = - K_0u \cdot g$ along a geodesic $\gamma(t)$:
\begin{equation}
\frac{d^2}{dt^2} (u \circ \gamma)(t) = \nabla du (\gamma'(t),\gamma'(t)) = - K_0(u \circ \gamma)(t)
\label{eq:ODE}
\end{equation}
Since $u$ is continuous across the conic points, it attains its maximum and minimum at points that we denote
$p_{\max}$ and $p_{\min}$. We have shown that $du$ vanishes at the conic points, so $du(p_{\max}) = du(p_{\min}) = 0$, 
regardless of whether $p_{\max/\min}$ is equal to one of the $p_j$ or not.
Since $u$ is nontrivial, $u(p_{\max}) > 0$ and $u(p_{\min}) < 0$. 

Now connect $p_{\max}$ to $p_{\min}$ by a minimizing geodesic $\gamma:[0,\ell] \rightarrow \Si$. 
Since $\Si$ is an Alexandrov space, $\gamma$ avoids the singular points  except possibly at
its endpoints. Multiply $u$ by a positive constant so that $u(p_{\max}) = 1$. Then the solution of \eqref{eq:ODE} is
\[
(u \circ \gamma)(t) = \cos (\sqrt{K_0}\,t).
\]
Since $du(p_{\min})=0$, we must have $\ell \geq\pi/\sqrt{K_0}$. Combining this with the diameter estimate in
Theorem \ref{th:diam}, we conclude that $\ell =\pi/\sqrt{K_0}$.

We have now obtained two points $p_{\max}, p_{\min} \in \Si$ with $\mbox{dist}\,(p_{\max}, p_{\min} ) 
= \pi/\sqrt{K_0}$. The rigidity part of Theorem \ref{th:diam} now gives that $\Si$ is either a 
round sphere with constant curvature $K_0$ (so $p_{\max/\min}$ are smooth points) or else the
$K_0$-suspension of a circle of length $\theta<2\pi$ (in which case $p_{\max/\min}$ are the two
conic points). 
\end{proof}

An immediate corollary of this result is as follows. 
\begin{corollary}
If $(\Sip, g)$ is conic and has constant curvature $K_0$, then the Friedrichs extension of $P^g$ on $L^2$ 
is invertible except when $k=2$ and $(\Sip,g)$ is the spherical suspension as in the theorem above. 
\label{cor:invP}
\end{corollary}
\begin{proof}
If $K_0 \leq 0$, then $2P^g = \nabla^* \nabla - K_0 \geq 0$. If $K_0 < 0$, this inequality is strict and $P^g$ is invertible.
If $K_0 = 0$ and $\omega$ is in the nullspace, then integrating by parts gives that $\nabla \omega = 0$. However, by the 
indicial root computations above, any nontrivial solution of $P^g \omega=0$ must either vanish or blow up at each 
$p_j$, hence cannot be parallel. 

Now suppose that $K_0 > 0$. If $\Delta^g_1  = \nabla^* \nabla + K_0$ is the Hodge Laplacian on $1$-forms, 
then $2P^g = \Delta^g_1 - 2K_0$,  so our assertion is equivalent to the statement that the first nonzero eigenvalue 
of $\Delta^g_1$ is strictly larger than $2K_0$, except in the one special case.  But now recall that in two dimensions, the nonzero spectra
of $\Delta^g_0$ and $\Delta^g_1$ are the same because on the Friedrichs domain, $\delta^g$, $d$ and $\star$ commute 
with the Hodge Laplacian (on forms of degrees $0$, $1$ and $2$ collectively), so the result follows from the previous Proposition. 
\end{proof}

\section{The slice and deformation theorems}\label{conic_deformation_theory}
We now extend the results of \S 3.4 to surfaces with conic singularities.  As before we first study the local structure 
of the ungauged moduli space of conic constant curvature metrics and then construct a slice for the diffeomorphism action. 

\subsection{Spaces of conic constant curvature metrics and diffeomorphisms}
Fix a reference conic metric $g_0$ with smooth coefficients relative to the basis $dr^2$, $r dr dy$, $r^2 dy^2$. 

We begin by defining the space $\CM^{\ell,\delta, \nu}(\Sip)$ of conic metrics $g$ with coefficients in the $b$-H\"older 
space $\calC^{\ell,\delta}_b$, and such that near each cone point $p_j$, $|g - a_j g_0|_{g_0} \leq C r^\nu$ for some $a_j > 0$.
Fix mutually disjoint neighborhoods $\calU_j \ni p_j$ and cutoff functions $\chi_j \in \calC^\infty_0(\calU_j)$ with 
$\chi_j \equiv 1$ near $p_j$.  Set 
\[
g_0(a,\eta) :=\Bigl( 1 + \sum_j \chi_j  (a_j r^{2\eta_j}-1) \Bigr) g_0,  \quad a_j \in \RR^+,  \  \eta_j \in (-1-\beta_j, -\beta_j),
\]
which is a $2k$-dimensional family of conic metrics, and then define
\[
\CM^{\ell,\delta,\nu}(\Sip) = \{ g = g_0(a,\eta) + h :  h \in r^\nu \calC^{\ell,\delta}_b \}.
\]
Writing $h = f g_0 + h^0$ where $\tr^{g_0} h^0 = 0$, then 
\[
g = \Bigl(1 + \sum_j \chi_j (a_j r^{2\eta_j}-1) + f \Bigr)g_0 + h^0.
\]

Next consider the space of conic metrics with curvature $K_0 \in \RR$, 
\[
\CM_{K_0}^{\ell,\delta,\nu}(\Sip):=\{ g \in \CM^{\ell,\delta,\nu}(\Sip) : K^g = K_0 \},
\]
and their union, the space of conic metrics with (any) constant curvature 
\[
\CM_{\cc}^{\ell,\delta,\nu}(\Sip):=  \{ g \in \CM^{\ell,\delta,\nu}(\Sip) : K^g = \mathrm{const.} \} = \bigcup_{K_0 \in \RR}\CM_{K_0}^{\ell,\delta,\nu}(\Sip).
\]
There is an $\RR^+$ action on this space via metric scaling, but we typically normalize by demanding that the metrics have unit area, 
so that the genus and the cone angles determine $K_0$. 

A primary consideration below is the action of the relevant diffeomorphism group on this space of metrics. 
To this end, we consider the Banach Lie group $\Diff^{\ell+1,\delta, \nu+1}(\Sip)$.   It is only necessary to specify
the topology on a neighborbood $\calW$ of the identity in this group, and for this we define
\begin{multline*}
\Diff^{\ell+1,\delta, \nu+1}_b(\Sip) \supset \calW :=\bigl\{\exp(X) : X= \sum_j \chi_j (a_j r \del_r + b_j \del_y) + \tilde X :\\
\vec a,\vec b \in \RR^k, \tilde X \in r^{\nu+1}\calC^{\ell+1,\delta}_b(\Sip; T\Sip), \quad |\vec a| + |\vec b| + 
||\tilde X||_{b; \ell+1, \delta, \nu+1} < \epsilon \bigr\}.
\end{multline*}
Noting that $r\del_r$ and $\del_y$ are the dilation and rotation vector fields near each $p_j$, we see that locally near any $p_j$, any such 
diffeomorphism $F$ is the composition of a dilation and rotation and a diffeomorphism of a punctured disk which equals the identity
to order $\nu+1$. 

We shall also want to consider conic metrics in conformal form, and it is necessary to show that the various ways of representing these
metrics are equivalent. Recall that the classical existence of isothermal coordinates asserts that if $g$ is a  $\calC^{\ell,\delta}$ 
metric in some small ball, then there exists a local $\calC^{\ell+1,\delta}$ diffeomorphism $F$ and a function 
$\phi \in \calC^{\ell,\delta}$ such that $F^* g = e^{2\phi} |dz|^2$. The analogous theorem is true here.
\begin{proposition}
Suppose that $g \in \CM^{\ell,\delta,\nu}(\calU)$, where $\calU = \{w: |w| < 1\}$. Then there exist $F \in \Diff^{\ell+1,\delta,\nu+1}_b(\calU')$
and $\phi \in \RR \oplus r^\nu \calC^{\ell,\delta}_b$, both defined in some possibly smaller ball $\calU'$, so that
\[
F^* g = e^{2\phi} g_\be = e^{2\phi} (dr^2 + (1+\be)^2 r^2 dy^2).
\]
\label{confparam}
\end{proposition}
\begin{proof}
Let us first reduce to the case where the cone angle is $2\pi$ by setting $\hat{g} = |w|^{-2\beta}g$. This lies in $\CM^{\ell,\delta,\nu}(\calU)$
relative to $g_0 = dr^2 + r^2 dy^2$.  We then invoke the standard result about isothermal coordinates, but with the regularity
statement adapted to the conic setting.  There are two ways to do this.  The first is to choose the diffeomorphism $F(z) = w$ 
as a solution to the Beltrami equation
\[
\del_{\bar{z}} F =  \mu \del_{z} F,
\]
where $\mu$ is given by an explicit algebraic expression involving the coefficients of $\hat{g}$. That expression shows directly that
$\mu \in r^\nu \calC^{\ell,\delta}_b$.  This equation is solvable if $||\mu||_{L^\infty} < 1$, and we can ensure this by restricting
to a smaller ball $\calU'$ and dilating. It then follows by  conic elliptic theory (we do not spell out the details for this first order case) 
that $F \in \Diff^{\ell+1,\delta,\nu+1}_b(\calU')$, as claimed. Writing $F^* \hat{g} = e^{2\hat{\phi}} g_0$, 
then the conformal factor here is also given by an explicit algebraic expression in the coefficients of $\hat{g}$, from which it follows that
$\phi - \lambda \in r^\nu\calC^{\ell, \delta}_b$ for some $\lambda \in \RR$.

An alternate way to obtain $F$ is as follows: solve $\Delta_{\hat{g}}u = 0$ with $u(0) = 0$ and and $du|_0 \neq 0$, then choose $v$ 
so that $v(0) = 0$ and $dv = * du$, and finally set $F = u + i v$.  We must restrict to a smaller ball to ensure that $F$ is a diffeomorphism. 
The first order terms in the Laplacian are in $r^{\nu-1}\calC^{\ell-1, \delta}_b$, so the regularity theory from \S 4.3 gives that 
$F \in \Diff^{\ell+1,\delta,\nu+1}_b(\calU')$.

We have now proved that $F^* ( |w|^{-2\beta} g) = e^{2\hat{\phi}} |dz|^2$. Since $F^* |w| = e^{2\psi} |z|$ where
$\psi  - \lambda' \in r^\nu \calC^{\ell,\delta}_b$ for some $\lambda' \in \RR$, we can raise this to the power $-2\beta$ 
and transfer to the other side to conclude the proof. 
\end{proof}

We next prove a sharp regularity result. 
\begin{proposition}
Fix $g, \hat{g} \in \CM^{\ell,\delta,0}$ such that $e^{2\phi}g = \hat{g}$.  If $K^g, K^{\hat{g}} \in \calC^{\ell-2,\delta}_b(\Sip)$ and 
$g$ and $\hat{g}$ have the same cone angles, then $\phi \in \calD^{\ell,\delta}_{\mathrm{Fr}}(\Sip)$.
If $g$ and $\hat{g}$ are, in addition, both polyhomogeneous, then so is $\phi$. 
\label{sharpregres}
\end{proposition}
\begin{proof} The function $\phi \in \calC^{\ell,\delta}_b$ satisfies
\[
\Delta^{g} \phi -K^{g} + K^{\hat{g}} e^{2\phi} = 0,
\]
so $\Delta^{g} \phi \in \calC^{\ell-2,\gamma}_b$. This means that $\phi$ is in the H\"older-Friedrichs domain. 
If $g, \hat{g}$ are both polyhomogeneous, then $\phi \in \calD^{\ell,\delta}_{\mathrm{Fr}}$ 
for every $\ell \geq 0$. By Proposition~\ref{pr:impreg}, $\phi$ has a partial expansion up to order $r^2$ with a remainder
in $r^2 \calC^{\ell,\delta}_b$ for every $\ell$. Inserting this into the equation for $\phi$ shows that $\phi$ has
an expansion to order $4$. Continuing in this way, we see that $\phi$ has a complete expansion. 
\end{proof}

\begin{lemma}\label{lem:conseq_eigenbound}
Let $g_0 \in \CM_{K_0}^{\ell,\delta,\nu}(\Sip)$ for some $\nu > 0$ and if $K_0>0$, assume that $\Sip$ is not a 
$K_0$-suspension. Then for $m \leq \ell$, 
\[
\Delta^{g_0}-2K_0 : r^\nu \calC^{m,\delta}_b \to  r^{\nu-2} \calC^{m-2,\delta}_b
\]
is injective. Furthermore, if $f \in r^{\nu-2} \calC^{m-2,\delta}_b$, then $(\Delta^g-2K_0)u = f$ has a solution 
\[
u =  \sum_j \chi_j(a_j r + b_j \log r) + \tilde u
\]
with $a,b \in \RR^k$ and $\tilde u \in r^\nu \calC^{\ell,\delta}_b$.
\end{lemma}
\begin{proof}
Injectivity of $\Delta^{g_0} - K_0$ on functions which decay as $r \to 0$ is trivial when $K_0 \leq 0$. For $K_0>0$, if 
$(\Delta^g - 2K_0)u = 0$ and $u \in r^\nu \calC^{m,\delta}_b$, then $u \in \calD_{\mathrm{Fr}}$. Thus so long as $\Sip$ is not a 
$K_0$-suspension, injectivity follows by Proposition~\ref{eigenbound}.  Existence of a solution $u$ which
blows up at most logarithmically can be deduced by standard arguments (or extending the parametrix
discussion in \S 4 to include operators with finite regularity coefficients). The partial expansion of this solution is
then implied by Proposition~\ref{pr:impreg}.
\end{proof}

We can now prove the first theorem, about the structure of the entire space of constant curvature conic metrics.
\begin{theorem} If $\ell \geq 2$, then $\CM^{\ell,\delta,\nu}_\cc(\Sip)$ is a Banach submanifold of $\CM^{\ell,\delta,\nu}(\Sip)$. 
\end{theorem}
\begin{proof}
The argument is the same as in the smooth case.  Assuming that $g_0 \in \CM^{\ell,\delta,\nu}_{K_0}$, we wish to show that 
$\CM^{\ell,\delta,\nu}_\cc(\Sip)$ is a smooth manifold in a neighborhood of $g_0$. To this end, consider the map 
\[
K:  \CM^{\ell,\delta,\nu}(\Sip)  \longrightarrow r^{\nu-2} \calC^{\ell-2,\delta}_b(\Sip); \quad g \mapsto K^g. 
\]

The linearization at $g_0$ is
\begin{multline*}
\left. DK\right|_{a=1, \eta=0, h=0}(\dot{a}, \dot{\eta}, \dot{f}, \dot{h}^0) =  \\
\delta^{g_0}\delta^{g_0} \dot{h}^0 + \frac12 \left(\Delta^{g_0} - 2K_0\right)\Bigl(\sum_j \chi_j(\dot{a}_j + 2 \dot{\eta}_j \log r) + \dot{f}\Bigr).
\end{multline*}
By Lemma~\ref{lem:conseq_eigenbound}, the restriction of this to pure trace tensors
\[
DK^{g_0} :  \RR^k_{\dot{a}} \times \RR^k_{\dot{\eta}} \times r^\nu \calC^{\ell,\delta}_b 
\longrightarrow r^{\nu-2} \calC^{\ell-2, \delta}_b
\]
is surjective. By the implicit function theorem, some neighbourhood of $g_0$ in $\CM^{m,\delta,\nu}_\cc(\Sip)$ and some neighbourhood 
in $\CM^{m,\delta,\nu}_{K_0'}$ for $K_0'$ near $K_0$ is a smooth Banach submanifold.  
\end{proof}
\begin{remark}
An alternate proof (which could also have been used in the smooth case) proceeds by choosing an explicit smooth
submanifold of smooth conic metrics $g(w)$ which represent the space of all conformal classes near to $g_0 = g(0)$. 
The constant curvature metrics conformal to each $g(w)$ are obtained by solving the constant curvature equation,
and one then shows that the solutions $\phi(w)$ depend smoothly on $w$. 
\end{remark}

\subsection{The gauged moduli space} 
Now fix any polyhomogeneous element $g_0 \in \CM^{\ell,\delta,\nu}_\cc(\Sip)$,  and define
\[
S_{g_0} = \{ g_0(a,\eta) + f g_0 + R(\kappa):  a,\eta \in \RR^{k}, f \in r^{\nu}\calC^{\ell,\delta}_b, \kappa \in \Stt^\sing\},
\]
where $R$ is the regularizing map \eqref{regmap}. This replaces the subspace $\ker B^{g_0}$ in the smooth case;
notice that only the term $R(\kappa)$ here need not be in the nullspace of $B^{g_0}$ since
\begin{multline*}
B^{g_0} R(\kappa) = B^{g_0} \Bigl( \kappa - \calD^{g_0} (\sum_j \chi_j r^{-1 + \frac{1}{1+\be_j}} \omega_0^{(j)} )\Bigr) =  \\
- P^{g_0} \Bigl( \sum_j \chi_j r^{-1 + \frac{1}{1+\be_j}} \omega_0^{(j)} \Bigr).
\end{multline*}
Near $p_j$, this is  $\calO( r^{-2 + \frac{1}{1+\be_j}})$ rather than the expected $\calO( r^{-3 + \frac{1}{1+\be_j}})$ since 
$(\omega_0^{(j)}, r^{-1 + \frac{1}{1+\be_j}})$ is an indicial pair for $P^{g_0}$.  We can still define the slice
\[
\calS_{g_0,\e} = S_{g_0} \cap \CM^{\ell,\delta,\nu}_\cc(\Sip) \cap \calV,
\]
where $\calV$ is an $\varepsilon$-ball around $g_0$ in $\CM^{\ell,\delta,\nu}(\Sip)$. Clearly, 
\begin{multline*}
T_{g_0} S_{g_0} = \bigl \{ \Bigl(\sum_j \chi_j(\dot{a}_j + 2 \dot{\eta}_j \log r) +  f\Bigr)g_0 + R(\kappa) : \\
\dot a, \dot \eta \in \RR^{k},  f \in r^{\nu}\calC^{\ell,\delta}_b, \kappa \in \Stt^\sing\bigr\}.
\end{multline*}

\begin{proposition}
If $\calV$ is chosen small enough, then the intersection defining $\calS_{g_0, \e}$ is transverse, and hence
$\calS_{g_0,\e}$ is a smooth submanifold of $\CM^{\ell,\delta,\nu}$ with tangent space
\begin{multline*}
T_{g_0} \calS_{g_0,\e} = \Bigl \{ \Bigl( \sum_j \chi_j(\dot a_j + 2 \dot\eta_j \log r) + f\Bigr)g_0 + R(\kappa) : \\ 
\delta^{g_0}\delta^{g_0} R(\kappa) + \frac12 (\Delta^{g_0} - 2K_0)\Bigl(\sum_j \chi_j(\dot{a} + 2 \dot{\eta} \log r) + \dot{f}\Bigr) = 
\lambda \in \RR\Bigr\}. 
\end{multline*}
Furthermore, the map $T_{g_0} \calS_{g_0,\e} \longrightarrow \mathbb{R}^k \times \Stt^\sing \times \RR$, 
\begin{equation}
\Bigl(\sum_j \chi_j(\dot a_j + 2 \dot\eta_j \log r )+ \dot f\Bigr)g_0 + R(\kappa) \longmapsto (\dot \eta, \kappa,\lambda),
\label{projhS}
\end{equation}
is an isomorphism. Thus $\dim \calS_{g_0,\e} = 6\gamma-6 + 3k +1$. Finally, every element in $\calS_{g_0,\e}$ is polyhomogeneous. 
\end{proposition}
\begin{remark}
Geometrically, $\dot \eta \in \RR^k$ corresponds to an infinitesimal change of cone angles, $\kappa \in \Stt^\sing$ to an infinitesimal 
change of pointed conformal structure, and $\lambda \in \RR$ to an infinitesimal change of scale.  The assertion 
that \eqref{projhS} is an isomorphism means that $\dot{a}$ and $f$ are determined by the other parameters. Hence we recover 
the full set of expected infinitesimal deformations.
\end{remark}
\begin{proof}
The characterization of $T_{g_0} \calS_{g_0,\e}$ is obvious. Next, any $k \in T_{g_0}\CM^{\ell,\delta,\nu}$ decomposes as
\[
k = \Bigl( \sum_j \chi_j(\dot a_j + 2 \dot\eta_j \log r) + f_1\Bigr)g_0 + h^0, 
\]
where $f_1$ and $h^0$ lie in $r^\nu \calC^{m,\delta}_b$.  To prove the transversality, simply write
\[
k = \Bigl( \Bigl( \sum_j \chi_j(\dot a_j + 2 \dot\eta_j \log r) + f_2\Bigr)g_0 + h^0\Bigr) + f_3 g_0
\]
where the first term is tangent to $\CM^{\ell,\delta,\nu}_\cc$, so $\dot a_j$, $\dot \eta_j$ and $f_2$
are determined by the equation $\frac12 \Delta^{g_0} (\sum \chi_j (\dot a_j + 2 \dot \eta_j \log r) + f_2 )+  
\delta^{g_0}\delta^{g_0} h^0 = \lambda$, and then $f_3 = f_1 - f_2$ so that $f_3 g_0  \in S_{g_0}$. 
We conclude that $\calS_{g_0,\e}$ is a smooth submanifold. 

Now suppose that $(\sum \chi_j (\dot a_j + 2 \dot \eta_j \log r + f )g_0 + R(\kappa) \in T_{g_0} \calS_{g_0,\e}$. 
Given any $(\dot \eta, \kappa, \lambda)$, determine $\dot a$ and $\dot f$ by finding a bounded solution
(using the Friedrichs extension) to the equation
\begin{multline*}
\frac12 (\Delta^{g_0} - 2K_0) \bigl(\sum_j \chi_j \dot a_j + \dot f\bigr) g_0 \\
= \lambda - \delta^{g_0}\delta^{g_0} R(\kappa) 
- (\Delta^{g_0} - 2K_0) \bigl(\sum_j \chi_j \dot \eta_j \log r\bigr).
\end{multline*}
Since the Friedrichs extension is invertible, there is a unique solution, and this satisfies $\dot f \in r^\nu \calC^{\ell,\alpha}_b$.

The formula for the dimension of this space follows directly.

For the final assertion, suppose that $g = g_0(a,\eta) + fg_0 + R(\kappa) \in \calS_{g_0,\e}$. We compute that
\[
N^{g_0}(g) = (K^g - K_0)g + (\delta^g)^* B^{g_0}( g - g_0).
\]
Since $K^g = K_0$ and $B^{g_0}$ annihilates pure trace and transverse-traceless tensors, this reduces to 
$- (\delta^g)^* B^{g_0} \calD^{g_0} \omega = - (\delta^g)^* P^{g_0} \omega$ for some suitable $\omega$. 
However, we can certainly choose extensions of the terms $\omega_j^{(0)} r^{-1 + \frac{1}{1+\be_j}}$ so that they
are annihilated to all orders by $P^{g_0}$, which means that we can assume that this term vanishes to all orders at $p_j$.
In other words, we see that $N^{g_0}(g) = \eta$ which is smooth and vanishes to all orders at every $p_j$ 
The rest of the proof is now essentially the same as in Proposition \ref{sharpregres}.
\end{proof} 

We next show that the local action of our specified group of diffeomorphisms on the slice $\calS_{g_0,\e}$ generates a full neighbourhood in the space of conic constant curvature metrics. 

\begin{theorem}\label{thm:flowbox}
If $\calW \subset\Diff_b^{\ell+1,\delta, \nu+1}(\Sip)$ is a small neighborhood of the identity, then the map
\[
\calW \times \calS_{g_0,\e} \hookrightarrow \CM^{\ell,\delta,\nu}_{\cc}(\Sip), 
\qquad
(F, g) \longrightarrow F^* g,  
\]
is a diffeomorphism onto a neighbourhood of $g$ in $\CM^{\ell,\delta,\nu}_{\cc}(\Sip)$. 
\end{theorem}
\begin{proof}
This again reduces, using the inverse function theorem, to showing that if $k \in T_{g_0} \CM^{\ell,\delta,\nu}_\cc(\Sip)$,
then $k = (\sigma g_0 + R(\kappa)) + \delta^* \omega$, where $\sigma g_0 + R(\kappa) \in T_{g_0}\calS_{g_0}$ and 
$\omega =  \omega_1 r + \tilde \omega$ for some $ \tilde \omega\in r^{\nu+1}\calC^{\ell+1,\delta}_b(\Sip; T^*\Sip)$. 

The first step in doing this is to show that 
\[
k = \hat\sigma g_0 + \kappa + \delta^* \hat\omega,  \qquad \kappa \in \Stt^\sing.
\]
To find these summands, assume that this decomposition exists and apply $B^{g_0}$ to both sides, yielding
\[
P^{g_0} \hat\omega = B^{g_0} (\delta^{g_0})^* \hat\omega = B^{g_0} k.
\]
The right hand side lies in $r^{\nu-1}\calC^{\ell-1,\delta}_b$. Using Propositions \ref{pr:duality} and 
\ref{pr:impreg}, we obtain a solution with the partial expansion
\[
\hat{\omega} = \omega_0 r^{-1+\frac{1}{1+\beta}}+ \omega_1 r + \tilde{\omega}, \qquad \tilde{\omega} \in r^{\nu+1} \calC^{\ell+1,\delta}_b.
\]
Setting $\kappa=k^0-\calD^{g_0} \hat\omega$, then we have that $\delta^{g_0}(\kappa)=\tr^{g_0}(\kappa) = 0$, and in addition
$|\kappa| = \calO( r^{-1+\frac{1}{1+\beta}})$,  hence $\kappa\in\Stt^\sing$. 

Since $(\delta^{g_0})^* \hat\omega = \calD^{g_0}\hat\omega - \frac 12 (\delta^{g_0} \hat\omega)g_0$, we can thus write
\[
k = \hat\sigma g_0 + \kappa + (\delta^{g_0})^* \hat\omega
\]
where  $\hat\sigma = \tfrac 12( \tr^{g_0} k +  \delta^{g_0} \hat\omega)$.

To get the desired decomposition, set $\omega= \hat \omega - \omega_0 r^{-1+\frac{1}{1+\beta}}$. Then 
$\omega=\omega_1r+\tilde\omega$ as claimed. Furthermore, substituting $\hat \omega$ in this decomposition yields
\begin{align*}
k &= \sigma g_0 + \kappa + \calD^{g_0}(\omega_0 r^{-1+\frac{1}{1+\beta}})+ (\delta^{g_0})^* \omega\\
&= \sigma g_0 + R(\kappa) + (\delta^{g_0})^* \omega
\end{align*}
with $\sigma=\tfrac 12( \tr^{g_0} k +  \delta^{g_0} \omega)$. Note finally that 
$\kappa + \calD^{g_0}(\omega_0 r^{-1+\frac{1}{1+\beta}}) = \calO(r^\nu)$, hence equals $R(\kappa)$.
\end{proof}

We conclude this section by analyzing the action of the entire identity component of the diffeomorphism group on
the space of conic metrics with fixed cone angle. In the smooth case, this action is proper and free when $\chi(\Si) < 0$, 
see \cite[Section 2.2]{Tro}. We prove now that the same is true here. 

\begin{proposition}
Suppose that $\chi(\Sip) < 0$. 
Then the identity component $(\Diff_b^{\ell+1,\delta, \nu+1}(\Sip))_0$ acts freely on $\CM^{\ell,\delta,\nu}_\cc(\Sip)$ 
and $\Diff_b^{\ell+1,\delta, \nu+1}(\Sip)$ acts properly on $\CM^{l,\delta,\nu}_\cc(\Sip)$.	 
\end{proposition}
\begin{proof}
Suppose that $g_n, g_n'$ are two sequences of elements in $\CM^{\ell,\delta,\nu}_\cc(\Sip)$ which converge in the topology of 
this space to $g$ and $g'$, and $F_n$ is a sequence in $\Diff_b^{\ell+1,\delta, \nu+1}(\Sip)$ such that $F_n^* g_n = g_n'$.  
We must show that some subsequence of the $F_n$ converges in $\Diff_b^{\ell+1,\delta, \nu+1}(\Sip)$ 
to a diffeomorphism $F$, so that necessarily $F^* g = g'$.   

Tromba's proof of this in the smooth case proceeds as follows. First choose a finite net of points $\{q_i\}$ so that the
$\epsilon$ neighborhoods of these points cover $\Si$ and are normal neighborhoods for $g$, and hence for $g_n$ for $n$ large.
Next choose a subsequence $F_{n'}$ so that $F_{n'}(q_i)$ converges for every $i$, and set $q_i' = \lim_{n'\to\infty} F_{n'}(q_i)$.  
Pre- and post-composing with the exponential maps at $q_i$ and $q_i'$, and using that the $F_{n'}$ are isometries,
we see that $F_{n'}$ converges on all of $\Si$; however, due to the loss of derivatives in using the exponential maps,
this convergence is only in $\calC^{\ell-1,\delta}$.  The final step is to observe that the formula relating the 
second derivatives of $F_{n'}$ to the Christoffel symbols of $g_n$ and $g_n'$ allows one to conclude that the $F_{n'}$
converge in $\calC^{\ell+1,\delta}$. 

In the conic case, we may employ a very similar strategy.  The first step is the same, and we can choose a subsequence (which
we immediately relabel as $F_n$ again) which converges at the $\epsilon$-net $\{q_i\}$ and also at each of the cone points $p_j$. 
The argument in the smooth case localizes, and allows us to conclude that the $F_n$ converge in $\calC^{\ell+1,\delta}$ 
on the complements of $\epsilon$-balls around the $p_j$.  Hence it remains to prove the result in fixed balls around
an isolated cone point.  For this we use the conformal representations of these conic metrics, and write
\[
g_n = H_n^* ( e^{2\phi_n} g_0), \quad g_n' = (H_n')^* (e^{2\phi_n'} g_0).
\]
We have shown in Proposition~\ref{confparam} that $\phi_n - \lambda_n$ and $\phi_n' - \lambda_n'$ both converge
in $r^\nu \calC^{\ell,\delta}_b$, that $\lambda_n$ and $\lambda_n'$ both converge in $\RR$, and that $H_n, H_n'$ 
converge in $\Diff_b^{\ell+1,\delta,\nu+1}$. The identification between the model metrics $g_0$ in the domain and the range 
is fixed.  Set $\tilde{F}_n = H_n' \circ F_n \circ H_n^{-1}$, so that $\tilde{F}_n^* (e^{2\phi_n'} g_0) = e^{2\phi_n} g_0$. 

We now observe that $\tilde{F}_n: (\calU, e^{2\phi_n} g_0) \to (\calU, e^{2\phi_n'} g_0)$ is a harmonic mapping; this 
condition is conformally invariant, so we may as well replace the domain space by $(\calU, g_0)$, or even
by $(\calU, |dz|^2)$. Refined regularity results are available in this setting; we reference the ones proved
by Gell-Redman \cite{Gell-Redman} since they are closest in spirit to this paper. Precomposing with a convergent
family of rotations, $\tilde{F}_n(z) = z + v_n(z)$ where $v_n \in r^{\nu +1}\calC^{\ell+1,\delta}_b$; the Taylor expansion
of the harmonic map operator can be written as $L_n v_n = f_n + Q(v_n)$ where $Q(v_n)$ is a remainder
term which is quadratic in $\nabla v_n$ and $f_n$ is the inhomogeneous term which measures the extent
by which the identity map fails to be harmonic between $(\calU, |dz|^2)$ and  $(\calU, e^{2\phi_n'} g_0)$. 
By assumption, $f_n$ and $L_n$ converge in the appropriate topologies, so standard a priori estimates show
that some subsequence of the $v_n$ converges in $r^{\nu+1}\calC^{\ell+1,\delta}_b$. This finishes the proof of properness of the action.

It remains to prove freeness of the action of the identity component. To that end, let $F \in (\Diff_b^{\ell+1,\delta, \nu+1}(\Sip))_0$ 
and $g \in \CM^{l,\delta,\nu}_\cc(\Sip)$ be such that $F^*g=g$. Using Proposition~\ref{confparam} we may assume (after possibly 
applying a further element of $(\Diff_b^{\ell+1,\delta, \nu+1}(\Sip))_0$) that $g$ takes the model form $g_{\be_j,K}$ \eqref{models1} 
near each $p_j$. Writing $g_{\be_j,K}$ in terms of a local complex coordinate $z$, it is clear that $F$ extends to a diffeomorphism 
$\bar F$ of the closed surface which is isotopic to the identity and acts as a rotation near each $p_j$. 

Following \cite{Tro}, we employ the Lefschetz fixed point formula. Since $F$ is an isometry, its fixed points are either 
isolated or else $F$ and hence $\bar F \equiv \mathrm{id}$.  Assuming the former, let $p \in \Sip$ be a fixed point of 
$\bar F$. Then $\det (\id - d\bar F(p)) >0$ since $d \bar F(p)$ is a nontrivial rotation.  A similar calculation holds
near each $p_j$ with respect to the local complex coordinates there. Hence the Lefschetz number of $\bar F$ satisfies
\[
L(\bar F) = \sum_{p \in Fix(\bar F)} \sign \det (\id - d\bar F(p))\geq k.
\]  
However, since $\bar F$ is isotopic to the identity, $L(\bar F)=\chi(\Sigma)$ which is less than $k$ if either $\gamma=0$ 
and $k\geq 3$ or else if $\gamma \geq 1$ and $k \geq 1$.  This is a contradiction.
\end{proof}

\section{Deformation theory of complete hyperbolic metrics}
Our starting point is the classical existence theorem: if $\chi(\Sip) < 0$, then each pointed conformal class on $\Sip$ 
contains a unique complete  hyperbolic metric, and this metric has finite area.  A modern elementary proof using only 
basic elliptic PDE and barrier arguments appears in \cite{MT}, and see also \cite{JMS} for a proof based on Ricci flow. 
The Teichm\"uller space $\calT_{\gamma,k}$ is the quotient of the space of all such hyperbolic metrics by the group of 
diffeomorphisms of $\Si$ isotopic to the identity and fixing $\frakp$. The unpublished thesis \cite{Zeit} develops the local structure 
of this space in the style of \cite{Tro}. We now discuss briefly how the arguments above may be modified to go a bit further
than in \cite{Zeit} in this setting of complete finite area hyperbolic metrics. 

Let us set up some notation.  Just as in the conical setting, the ends of these solution metrics are all `standard',
and are isometric to a model hyperbolic cusp $( (0,1]_r \times S^1, g_c)$, where
\[
g_c = \frac{dr^2}{r^2} + r^2 dy^2,
\]
or in the equivalent conformal form $g_c = (|z| \log |z|)^{-2} |dz|^2$ using the change of variables $|z| = \rho = e^{-1/r}$. 
Now define $\HCM(\Sip)$ to be the space of all `asymptotically hyperbolic cusp' metrics $g$ on $\Sip$ which 
are asymptotic to $g_c$ near each $p_j$. More precisely, $\HCM^{\ell,\delta,\nu}(\Sip)$ consists of all metrics which 
are in $\calC^{\ell,\delta}$ away from the $p_j$ and which near these points have the form $g_c + h$ where $h \in \calC^{\ell,\delta,\nu}_\hc$ 
for some $\nu > 0$. This function space is defined as follows. Decompose the function $v$ on $[0,1]_r \times S^1_y$ as $v = v_0 + v_\perp$, 
where $v_\perp$ is the sum over all nonzero eigenmodes on $S^1$ and $v_0$ is independent of $y$.  Fixing any $c \in (0,1)$, set 
\[
\calC^{\ell,\delta,\nu}_\hc = \left\{ v: v_0 \in r^\nu \calC^{\ell,\delta}_b([0,1]), \quad v_\perp \in e^{-c/r} \calC^{\ell,\delta}_b([0,1]\times S^1) \right\}. 
\]
We are interested in the subspace $\HCM^{\ell,\delta,\nu}_{-1}(\Sip)$ of hyperbolic metrics. 

These function spaces are well suited to the solvability properties of the scalar operator $\Delta^g + 2$ appearing in the linearization
of $DK$ at a hyperbolic metric. 
\begin{proposition}
If $g \in \HCM^{m,\delta,\nu}(\Sip)$, then 
\[
\Delta^g + 2:  \calC^{m,\delta,\nu}_\hc(\Sip) \longrightarrow \calC^{m-2,\delta,\nu}_{\hc}(\Sip)
\]
is Fredholm provided $\nu \neq -2, 1$. It is an isomorphism when $\nu \in (-2,1)$. 
\label{cuspfr}
\end{proposition}
\begin{proof}
It suffices to produce an inverse for this operator on any cusp end, $\calE := \{0 < r \leq r_0\} \times S^1$, say with Dirichlet
boundary conditions.  Given $f \in \calC^{m-2,\delta,\nu}_{\hc}(\calE)$, we wish to produce a solution $u \in \calC^{m,\delta,\nu}_{\hc}$,
at least up to a finite rank error.  In $\calE$, we have
\[
\Delta^{g_c} + 2= -r^2 \del_r^2 - 2 r\del_r - r^{-2} \del_y^2 + 2.
\]
The induced equation on the zero mode is of Euler type, and it is straightforward to see that this one-dimensional 
operator has closed range when $\nu \neq -2, 1$, is an isomorphism when $\nu \in (-2,1)$, is surjective when 
$\nu > 1$ and injective when $\nu < -2$. 

On the other hand, consider the induced operator on the orthogonal complement of the zero mode.  We 
decompose further into the different eigenmodes and solve the separate ODE's 
\begin{equation}
L_j u_j := \left(-r^2 \del_r^2 -2 r\del_r + \frac{j^2}{r^2} + 2\right) u_j = f_j, \quad |j| \geq 1.
\label{redeqn}
\end{equation}
It is not hard to produce a unique decaying solution for each of these equations which vanishes at $r=r_0$, 
so the problem is to obtain a uniform rate of decay and to sum over $j$. For this, note that
\[
L_j e^{-c/r} = \left(2 + \frac{j^2-c^2}{r^2} \right) e^{-c/r} \geq c' e^{-c/r} > 0
\]
since $j^2 - c^2 \geq 1-c^2 > 0$ and $r < r_0$, so $e^{-c/r}$ is a supersolution.  Suppose now that $|f_j(r)| \leq a e^{-c/r}$
for all $j$. Then, for some $A > 0$ depending on $a$ and $c'$, 
\[
L_j ( u_j - \frac{A}{j^2} e^{-c/r} ) < 0, \qquad L_j (\frac{A}{j^2} e^{-c/r} - u_j ) > 0,
\]
which implies that $|u_j| \leq A j^{-2} e^{-c/r}$, and hence that $|u_\perp| \leq A' e^{-c/r}$. 

Local elliptic estimates (applied on the universal cover of $\calE$, for example) now show that $u \in \calC^{m,\delta,\nu}_{\hc}(\calE)$,
as required.

These local parametrices may be patched together in the usual way to obtain a global parametrix which acts on these
function spaces, and which is inverse to $L$ up to compact error. 
\end{proof}

When $\nu \in (-2,1)$, this map is an isomorphism. We now fix any $\nu \in (-2,0)$; we restrict $\nu$ to be negative
because of the nonlinearity of the problem. 
\begin{proposition}
The subspace $\HCM^{m,\delta, \nu}_{-1}(\Sip)$ is a Banach submanifold of $\HCM^{m,\delta,\nu}(\Sip)$.
\end{proposition}
The proof is the same as in the smooth and conic cases. Suppose that $g \in \HCM^{m,\delta,\nu}_{-1}$. We must
show that 
\[
DK^g: \calC^{m,\delta,\nu}_{\hc}(\Sip) \longrightarrow \calC^{m-2,\delta,\nu}_{\hc}(\Sip) 
\]
is surjective. Restricting to $h = fg$, this reduces to the surjectivity of $\Delta^g + 2$, which is
what we have just proved. 

Next, define the slice
\begin{equation}
\calS_{g} = ((1 + f)g + \Stt^{\sing}) \cap \HCM^{m,\delta,\nu}_{-1}(\Sip) \cap \calV, \quad f \in \calC^{m,\delta,\nu}_{\hc},
\label{slicehc}
\end{equation}
where $\calV$ is a small neighborhood of $g$. Notice that no regularization of $\kappa \in \Stt^{\sing}$ is needed 
because for any such $\kappa$, $|\kappa|_g \leq C e^{-1/r} r^{-2}$, which lies in $\calC^{m,\delta,\nu}_{\hc}$. 
\begin{proposition}
The intersection \eqref{slicehc} is transverse, so $\calS_g$ is a smooth finite dimensional submanifold of
dimension $6\gamma - 6 + 2k$. 
\end{proposition}
\begin{proof}
Suppose that $f_1 g + h_1^0 \in T_g \HCM^{m,\delta,\nu}$; we wish to write this as the sum $( f_2 g + h_2^0) + (f_3 g + \kappa)$,
where $\delta^g \delta^g h_2^0 + (\Delta^g + 2)f_2 = 0$ and $\kappa \in \Stt^\sing$.  We can then let $h_2^0 = h_1^0$, 
determine $f_2$ by the equation, and let $\kappa = 0$ and $f_3 = f_1 - f_2$.  To see that the dimension is
correct, note that if $\kappa \in \Stt^\sing$, then there is a unique $\phi$ such that 
$e^{2\phi}(g + \kappa) \in \HCM^{m,\delta,\nu}(\Sip)$. 
\end{proof}

Now define $\Diff_{\hc}^{\ell+1,\delta, \nu}(\Sip)$ in a similar way as we did in the conic case. (Note
that the weight $\nu$ is not shifted to $\nu+1$ here.)   Let $\calW$ be a neighborhood of the identity in this group.
\begin{proposition}
Assume that the base metric $g$ is polyhomogeneous. Then for sufficiently small neighborhoods $\calV$ and $\calW$, the map 
\[
\calW \times \calS_{g} \longrightarrow \HCM^{\ell,\delta,\nu}_{-1}(\Sip), \qquad (F, h) \mapsto F^*( g+h),
\]
is a diffeomorphism onto a small neighborhood of $g$ in $\HCM^{m,\delta,1}_{-1}$. 
\end{proposition}
\begin{proof}
We check that any $f g + h^0$ with $\delta^g \delta^g h^0 + (\Delta^g + 2)f = 0$ can be written
as $\kappa + (\delta^g)^* \omega$. We find $\omega$ by solving $P^g \omega = B^g (\delta^g)^* \omega = B^g h^0$.
This is accomplished by a direct analogue of Proposition~\ref{cuspfr} which we leave to the reader. 
Then $\kappa =  h^0 - \calD^g \omega$ lies in $\Stt^\sing$. The rest is clear.  
\end{proof}

 \section{The full moduli space} 
Let $\CM_{\cc,1}(\Sip)$ denote the space of conic constant curvature metrics normalized to have unit area. 
By the Gauss-Bonnet theorem \eqref{eq:GB2} 
$K^g$ is zero, negative or positive if and only if $\vec\beta \in \mathsf{Euc}$, $\mathsf{Hyp}$, and $\mathsf{Sph}$, respectively. 

The Teichm\"uller space of conic constant curvature metrics by definition is $\CM_{\cc,1}(\Sip)$ modulo diffeomorphisms isotopic 
to the identity: 
\[
\calT^{conic}_{\gamma,k}= \CM^{\ell, \delta,\nu}_{\cc,1}(\Sip) / (\Diff_b^{\ell+1,\delta, \nu+1}(\Sip))_0.
\]
This is a smooth manifold of dimension $6\gamma-6 + 3k$ according to the results in Section \ref{conic_deformation_theory}. 
The standard Teichm\"uller space $\calT_{\gamma,k}$ is a ball of dimension $6\gamma-6 + 2k$. Assigning to $g$ its 
cone angle parameters $\vec\beta \in(-1,0)^k$, we obtain a smooth map 
\[
B : \calT^{conic}_{\gamma,k} \to (-1,0)^k.
\]	
The fiber $B^{-1}(\vec \beta)$ (if non-empty) is identified with $\calT_{\gamma,k}$ in the following way.  An element $g \in \CM^{\ell,\delta,\nu}$
defines a conformal structure of finite type on $\Sigma\setminus \mathfrak p$, hence an element in $\calT_{\gamma,k}$. Conversely, 
by \cite{Tr}, \cite{Mc}, each conformal structure in $\calT_{\gamma,k}$ is represented by a unique conic constant curvature metric with 
cone angles parameters $\vec\beta$. The standard identification of $\calT_{\gamma,k}$ with the space of complete, finite area
hyperbolic metrics on $\Sigma \setminus \mathfrak p$ is the case where all cone angles are zero. By the characterization of $T_{g_0}\calS_{g_0,\e}$ in \eqref{projhS}, the differential of $B$ is surjective at every point in $\calT^{conic}_{\gamma,k}$. We thus obtain the following

\begin{proposition}
The map
\[
B: \calT^{conic}_{\gamma,k} \to (-1,0)^k
\]	
is a submersion, hence the fiber $B^{-1}(\vec\beta)$ is a submanifold for any $\vec\beta \in (-1,0)^k$. The image of $B$ is the 
region $\mathsf{Hyp}$ if $\gamma \geq 1$ and the region $\mathsf{Hyp} \cup \mathsf{Euc} \cup \mathsf{Sph}$ if $\gamma=0$.
\end{proposition}

\begin{remark}
With a little more effort one may show that $\calT^{conic}_{\gamma,k}$ is diffeomorphic to $\mathsf{Hyp} \times \calT_{\gamma,k}$ 
if $\gamma\geq 1$, resp.\ to $\mathsf{Hyp} \cup \mathsf{Euc} \cup \mathsf{Sph} \times \calT_{\gamma,k}$ if $\gamma=0$.
\end{remark}

This is particularly interesting in the case $\gamma=0$, i.e.\ on the 2-sphere. According to the above decomposition of the 
Troyanov region we obtain a corresponding decomposition of the Teichm\"uller space of conic metrics
\[
\calT^{conic}_{0,k} = B^{-1}(\mathsf{Hyp}) \cup B^{-1}(\mathsf{Euc}) \cup B^{-1}(\mathsf{Sph}).
\]
Here $B^{-1}(\mathsf{Hyp})$ and $B^{-1}(\mathsf{Sph})$ are both open whereas the Euclidean structures $B^{-1}(\mathsf{Euc})$ 
determine a hypersurface separating the spherical from the hyperbolic structures. In particular, any Euclidean cone structure 
on the 2-sphere possesses deformations into spherical as well as hyperbolic ones. This phenomenon has been studied 
before in the 3-dimensional context in \cite{PorWei}, and in fact the regeneration of a Euclidean structure (obtained as the 
collapsed limit of a sequence of hyperbolic structures) into a spherical one is an important step in the proof of Thurston's 
Orbifold Theorem. 



\vskip1in

\begin{thebibliography}{99999}

\bibitem{Be}
{\sc A.L.\ Besse},
{\em Einstein manifolds}, 
Ergebnisse der Mathematik und ihrer Grenzgebiete (3), 10. Springer-Verlag, Berlin, 1987. 

\bibitem{BLP}
{\sc M.\ Boileau, B.\ Leeb, J.\ Porti}, 
{\em Geometrization of 3-dimensional orbifolds}, 
Ann.\ of Math.\ 162 (2005), no.\ 1, 195--290.

\bibitem{BBI}
{\sc D.\ Burago, Y.\ Burago, S.\ Ivanov},
{\em A course in metric geometry}, 
Graduate Studies in Mathematics, 33. American Mathematical Society, Providence, RI, 2001.

\bibitem{BGP}
{\sc Y.\ Burago, M.\ Gromov, G.\ Perelʹman},
{\em A.D.\ Aleksandrov spaces with curvatures bounded below} (Russian), Uspekhi Mat.\ Nauk 47 (1992), no.\ 2(284), 3--51, 222; translation in Russian Math. Surveys 47 (1992), no.\ 2, 1--58. 

\bibitem{CM1}
{\sc A.\ Carlotto, A.\ Malchiodi}, 
{\em A class of existence results for the singular Liouville equation}, 
C.\ R.\ Math.\ Acad.\ Sci.\ Paris 349 (2011), no.\ 3-4, 161--166.

\bibitem{CM2}
{\sc A.\ Carlotto, A.\ Malchiodi},
{\em Weighted barycentric sets and singular Liouville equations on compact surfaces}, 
J.\ Funct.\ Anal.\ 262 (2012), no.\ 2, 409--450.

\bibitem{Don}
{\sc S.K.\ Donaldson}, 
{\em K\"ahler metrics with cone singularities along a divisor}, 
Essays in mathematics and its applications, 49--79, Springer, Heidelberg, 2012. 

\bibitem{Er}
{\sc A.\ Eremenko},
{\em Metrics of positive curvature with conic singularities on the sphere}, 
Proc.\ Amer.\ Math.\ Soc.\ 132 (2004), no.\ 11, 3349--3355 (electronic).

\bibitem{JMR}
{\sc T.\ Jeffres, R.\ Mazzeo, Y.\ Rubinstein},
{\em K\"ahler-Einstein metrics with edge singularities} (with an Appendix by C.\ Li and Y.\ Rubinstein),
{\sf arXiv:1105.5216}, to appear in Ann.\ of Math. 

\bibitem{JMS}
{\sc L.\ Ji, R.\ Mazzeo, N.\ Sesum},
{\em Ricci flow on surfaces with cusps}, 
Math.\ Ann.\ 345 (2009), no.\ 4, 819--834.

\bibitem{Gell-Redman}
{\sc J.\ Gell-Redman}, 
{\em Harmonic maps of conic surfaces with cone angles less than $2\pi$},
Comm. Anal. Geom. 23 (2015), no. 4, 717--796.

\bibitem{Kok}
{\sc A.\ Kokotov},
{\em Polyhedral surfaces and determinant of Laplacian},
Proc.\ Amer.\ Math.\ Soc.\ 141 (2013), no.\ 2, 725--735. 

\bibitem{LT}
{\sc F.\ Luo, G.\ Tian}, 
{\em Liouville equation and spherical convex polytopes}, 
Proc.\ Amer.\ Math.\ Soc.\ 116 (1992), no.\ 4, 1119--1129.

\bibitem{Ma}
{\sc R.\ Mazzeo}, 
{\em Elliptic theory of differential edge operators I}, 
Comm.\ Partial Differential Equations 16 (1991), no.\ 10, 1615--1664.

\bibitem{MT}
{\sc R.\ Mazzeo, M.\ Taylor},
{\em Curvature and uniformization},
Israel J.\ Math.\ 130 (2002), 323--346.
 
\bibitem{Mc}
{\sc R.C.\ McOwen}, 
{\em Point singularities and conformal metrics on Riemann surfaces}, 
Proc.\ Amer.\ Math.\ Soc.\ 103 (1988), no.\ 1, 222--224.

\bibitem{MRS}
{\sc R.\ Mazzeo, Y.A.\ Rubinstein, N.\ Sesum},
{\em Ricci flow on surfaces with conic singularities},
Anal.\ PDE 8 (2015), no.\ 4, 839--882.

\bibitem{MonP}
{\sc G.\ Mondello, D.\ Panov},
{\em Spherical metrics with conical singularities on a 2-sphere: angle constraints},
{\sf arXiv:1505.01994}.

\bibitem{PSS}
{\sc D.H.\ Phong, N.\ Sesum, J.\ Sturm},
{\em Multiplier ideal sheaves and the K\"ahler-Ricci flow},
Comm.\ Anal.\ Geom.\ 15 (2007), no.\ 3, 613--632. 

\bibitem{PSSW}
{\sc D.H.\ Phong, J.\ Song, J.\ Sturm, X.\ Wang},  
{\em The Ricci flow on the sphere with marked points},
{\sf arXiv 1407.1118}.

\bibitem{PorWei}
{\sc J.\ Porti, H.\ Weiss},
{\em Deforming Euclidean cone 3-manifolds},
Geom.\ Topol.\ 11 (2007), 1507--1538.

\bibitem{ST}
{\sc G.\ Schumacher, S.\ Trapani},
{\em Weil-Petersson geometry for families of hyperbolic conical Riemann surfaces}, 
Michigan Math.\ J.\ 60 (2011), no.\ 1, 3--33. 

\bibitem{Tro}
{\sc A.J.\ Tromba},
{\em Teichm\"uller theory in Riemannian geometry}, 
Lecture notes prepared by Jochen Denzler. Lectures in Mathematics ETH Z\"urich. Birkh\"auser Verlag, Basel, 1992.

\bibitem{Tr}
{\sc M.\ Troyanov},
{\em Prescribing curvature on compact surfaces with conical singularities}, 
Trans.\ Amer.\ Math.\ Soc.\ 324 (1991), no.\ 2, 793--821. 

\bibitem{UY}
{\sc M.\ Umehara, K.\ Yamada}, 
{\em Metrics of constant curvature 1 with three conical singularities on the 2-sphere},
Illinois J.\ Math.\ 44 (2000), no.\ 1, 72--94. 

\bibitem{Zeit}
{\sc T.\ Zeitlh\"ofler},
{\em A Poincar\'e theorem and a slice theorem for Teichm\"uller theory of punctured surfaces}, 
Thesis (Ph.D.)–University of California, Santa Cruz. 1999.

\end{thebibliography}
\end{document}